\documentclass[12pt]{amsart}
\usepackage{epstopdf, amsxtra, pifont, pdflscape}
\usepackage{url}
\usepackage{graphicx}
\usepackage{tikz}
\usepackage[all]{xy}
\usepackage[utf8]{inputenc}
\usetikzlibrary{decorations.markings,arrows.meta}
\newtheorem{definition}{Definition}
\newtheorem{state}{Proposition}

\newtheorem{conjecture}{Conjecture}
\newtheorem{cor}{Corollary}

\newtheorem{lem}{Lemma}
\usepackage{mathtools}
\DeclarePairedDelimiter\ceil{\lceil}{\rceil}
\DeclarePairedDelimiter\floor{\lfloor}{\rfloor}
\newcommand{\cl}[1]{\ceil*{\frac{#1}{2}}}
\newcommand{\fl}[1]{\floor*{\frac{#1}{2}}}
\newcommand{\qnum}[1]{\begin{bmatrix}
#1
\end{bmatrix}_q
}
\newcommand{\qbinom}[2]{\begin{bmatrix}
#1 \\
#2
\end{bmatrix}_q
}
\newcommand{\qbinomm}[2]{\begin{bmatrix}
#1 \\
#2
\end{bmatrix}_{q^2}
}

\title[$q$-super Catalan numbers]%
{$q$-Super Catalan Numbers: Combinatorial identities, Generating Functions,
and Narayana Refinements}

\begin{document}

\author{Arthur Rodelet--Causse}
\address[Arthur Rodelet--Causse]{LIGM, Universit\'e Gustave-Eiffel, CNRS,
ENPC, ESIEE-Paris \\
5 Boulevard Descartes \\Champs-sur-Marne \\77454 Marne-la-Vall\'ee cedex 2 \\
FRANCE}
\email{arthur.rodelet-causse@univ-eiffel.fr}

\author{Lenny Tevlin}
\address[Lenny Tevlin]{No longer affiliated with an academic institution}
\email{ltevlin@gmail.com}

\maketitle
\tableofcontents

\section{Introduction}

The super Catalan numbers
\begin{equation}
\label{def:sC}
T_{n,m}= \frac{(2m)!(2n)!}{(m)!(n)!(m +n)!}. \\
\end{equation}
are a two-parameter generalization of the central binomial numbers (at $m=0$)
and twice the usual Catalan numbers (at $m=1$).
They were first introduced by Catalan himself~\cite{Catalan} and studied by
several authors at the beginning of the 20th century.
The interest in them was revived much later by Ira Gessel in \cite{Gessel92},
who pinned the name super-Catalan and pointed out that the $T_{n,2}$ are
(twice) superballot numbers. While there have been numerous attempts at
finding a combinatorial (or otherwise) interpretation of all super Catalan
numbers, see \emph{e.g.}, \cite{Gessel}, \cite{Scha}, \cite{Chen-Wang},
\cite{Allen2014}, \cite{Limanta}, the question is still open and we do not
address it here.

Our object of interest is their natural $q$-analogs, the $q$-super
Catalan numbers defined by Warnaar and Zudilin in~\cite{Warnaar2011}:
\begin{equation}
T_{n,m}(q) = \frac{\qnum{2m}! \qnum{2n}!}{\qnum{m}! \qnum{n}! \qnum{m +n}!},
\end{equation}
where
\begin{equation}
\qnum{k}! = \prod_{i=1}^{k} \qnum{i}
\text{\qquad and\qquad}
\qnum{i} = \frac{1-q^i}{1-q}.
\end{equation}

In particular, the family includes the usual $q$-central binomial coefficients
and the $q$-Catalan numbers (up to a factor $(1+q)$):
\begin{equation}
T_{n,0}(q) = \qbinom{2n}{n} \qquad\text{and}\qquad
T_{n,1}(q) = (1+q)C_n(q),
\end{equation}
where
\begin{align}
& \qbinom{n}{k}  = \frac{\qnum{n}!}{\qnum{k}! \qnum{n-k}!},  \\ 
& C_n(q) = \frac1{\qnum{n+1}}\qbinom{2n}{n}
\end{align}

Since the $q$-super Catalan numbers are symmetric with respect to the
interchange of $n$ and $m$, we will assume that $n \geq m$ in the paper unless
specified otherwise.

Little is known in general about those numbers except they are
positive polynomials in $q$, see~\cite{Warnaar2011} and that they can be
decomposed into sums of generalizations of Narayana numbers
(see~\cite{Wachs}).
Our goal is to emphasize the commonality in this family of $q$-analogs. It is
usually believed that among all possible $q$-analogs of ordinary Catalan
numbers, a 'combinatorial' one would enjoy a $q$-version of Segner relation,
while the classical analog of ordinary Catalan numbers, $C_n(q)$, does not. We
will provide such an analog in Section~\ref{sec:convo} as well as the
resulting generating function for (scaled) $C_n(q)$ in Section~\ref{sec-genf}.

\subsection*{Outline of the paper}

After recalling the usual notations on $q$-analogs and $q$-series,
we begin by deriving a number of combinatorial identities satisfied by the
$q$-super Catalan numbers. In particular, we extend some of the known
combinatorial identities (Touchard, Koshy, Reed Dawson) to the $q$-super
Catalan numbers.

Next, we introduce some $q$-convolution identities involving q-central
binomial and q-Catalan numbers and derive a generating function for
$q$-Catalan numbers.
Then we introduce Narayana-type refinements of the super Catalan numbers.
We prove algebraically the $\gamma$-positivity of those refinements and give a
combinatorial proof in a special case through the type B analog of noncrossing
partitions.
Then we introduce their natural $q$-analogs, prove their
$q$-$\gamma$-positivity and prove some identities they satisfy, generalizing
identities of Kreweras~\cite{Krew} and Le Jen-Shoo~\cite{Gould}.
Using yet another identity, we prove that these refinements are
positive integer polynomials in $q$.

\subsection*{Usual notations on $q$-series}

We will use the following standard notations:

\begin{align}
& (a;q)_n = 
\begin{cases}
\prod_{k=0}^{n-1} (1 - a q^k) & \text{for\ } n \geq 1, \\
1 & \text{for\ } n = 0,
\end{cases} \\
&  {}_2 \phi_1(a,b;c;q;z)
    = \sum_{k =0}^{\infty} \frac{(a,q)_k(b;q)_k}{(q;q)_k(c;q)_k} z^k .
\end{align}
With these notations, $q$-binomial coefficients and $q$-super Catalan numbers
can be written as
\begin{align}
\label{eq:binom}
& \qbinom{n}{k}  = \frac{(q^{n-k+1};q)_k}{(q;q)_k},  \\ 
\label{eq:super}
& T_{n,m}(q) = \frac{( q ; q )_{2 m} ( q ; q )_{2 n}}{ ( q ; q )_m ( q ; q )_n
( q ; q )_{n + m}}.
\end{align}
Furthermore, we will use the following $q$-series identities in the proofs:
\begin{align}
\label{eq:Kone}
& ( q ; q )_n = ( q ; q )_m ( q^{m +1} ; q )_{n-m} \text{\ if $n\geq m\geq0$}\\
\label{eq:Ktwo}
 & ( q^2 ; q^2 )_n = ( -q ; q )_n ( q ; q )_n \\
 \label{eq:Kthree}
 & ( q ; q )_n = ( q ; q^2 )_{\left\lceil \frac{n}{2} \right\rceil} ( q^2 ; q^2 )_{\fl{n}} \\
 \label{eq:Kfour}
 & ( q^a ; q^b )_n = (-1)^n q^{an + b\frac{n(n-1)}{2}} (q^{-a-b(n-1)};q^b)_n
\end{align}

\section{$q$-super Catalan Identities}

Among all identities involving $q$-Catalan numbers and some identities
involving $q$-super Catalan numbers, we decided to focus first on the
Touchard identity involving Catalan numbers since it was generalized in both
directions respectively by Andrews~\cite{Andrews2010} and both
Gould~\cite{Gould} and Gessel-Xin~\cite{Gessel}. We will provide a formula
that unifies all previously known formulas, showing the first commonality
between all $q$-super Catalan numbers.

We will also derive generalizations of other $q$-identities derived by Andrews
in two different papers, one generalizing an identity by Koshy, the other an
identity by Reed Dawson.

\subsection{A $q$-Touchard identity for the $q$-super Catalan numbers}
\label{sec:Touchard}

The original identity by Touchard involves Catalan numbers and is
\begin{equation}
C_{n+1} = \sum_{k = 0}^{\fl{n}} 2^{n-2k} \binom{n}{2k} C_{k}.
\end{equation}

It was generalized by Andrews in~\cite{Andrews2010} who proved that
\begin{equation}
C_{n+1}(q)
 = \sum_{k = 0}^{\fl{n}}
      q^{2k(k+1)} \qbinom{n}{2k} \frac{(-q^{k+2};q)_{n-k}}{(-q;q)_k} C_k(q),
\end{equation}
while for ordinary super Catalan numbers it is known~\cite{Gould},
\cite{Gessel} that
\begin{equation}
\label{id:initial}
T_{n+m,m} = \sum_{k=0}^{\fl{n}} 2^{n-2k} \binom{n}{2k} T_{k,m}.
\end{equation}

We shall prove a $q$-version of \eqref{id:initial} that extends the result
of Andrews to all $q$-super Catalan numbers, namely:
\begin{state}
\label{prop:q-Touchard}
Let $m$ and $n$ be nonnegative integers. We have the following formula
that is the $q$-analog of the generalization of the Touchard identity to all
super Catalan numbers\footnote{Here and below we mark in
$\textcolor{red}{red}$ the adjustments needed in the identities to accommodate
the 'super' parameter $m$.}:
\begin{equation}
\label{eq:q-Touchard}
T_{n+m,m}(q) =
  \sum_{k = 0}^{\fl{n}} q^{2k(k+\textcolor{red}{m})} \qbinom{n}{2k}
    \frac{(-q^{k+\textcolor{red}{m}+1};q)_{n-k}}{(-q;q)_k} T_{k,m}(q).
\end{equation}
\end{state}
Our proof of this identity closely follows the logic of the original proof in
\cite{Andrews2010} and depends on Gauss summation formula
($q$-Chu-Vandermonde) for $q$-generalized hypergeometric functions, which we
record here for reference (II.7), \cite{GR}:
\begin{equation}
\label{eq:Gauss1} 
 {}_2 \phi_1\left(a,q^{-N};c;q;\frac{c q^N}{a}\right)
 = \frac{(\frac{c}{a};q)_N}{(c;q)_N} , 
\end{equation}
where $N$ is a nonnegative integer.

\begin{proof}
Consider the following sum:
\begin{equation}
\label{eq:step-touch}
 \sum_{k = 0}^{\fl{n}}
   \frac{q^{2k(k+m)}}{(-q;q)_k(-q^{m+1};q)_{k}} \qbinom{n}{2k} T_{k,m}(q).
\end{equation}
We first write out a generic term in the sum using \eqref{eq:binom} and
\eqref{eq:super}, cancel common factors, and then rewrite it using
~\eqref{eq:Kone} and~\eqref{eq:Ktwo}:
\begin{align*}
& \frac1{(-q;q)_k(-q^{m+1};q)_{k}}  \qbinom{n}{2k}  T_{k,m}(q)\\
& = \frac{(q;q)_{2m}}{(q;q)_m}  \frac{(q^{n - 2k +1};q)_{2k}}{(-q;q)_k (q;q)_k(q;q)_m (q^{m+1};q)_k (-q^{m+1};q)_k} \\
& =  \frac{(q;q)_{2m}}{(q;q)_m^2 }  \frac{(q^{n - 2k+1};q)_{2k}}{ (q^2;q^2)_k
(q^{2m+2};q^2)_{k}} .
\end{align*}

Now, split the numerator into even and odd terms followed by the reversal of powers using~\eqref{eq:Kfour}:
\begin{align*}
& (q^{n - 2k +1};q)_{2k} =  (q^{2\fl{n} - 2k+2};q^2)_k (q^{2\cl{n} - 2k +1};q^2)_k  \\
&  (q^{2\fl{n} - 2k+2};q^2)_k = (-1)^k q^{k(2\fl{n} - 2k+2) +k(k-1)} (q^{-2\fl{n}};q^2)_k \\
 & (q^{2\cl{n} - 2k+1};q^2)_k =  (-1)^k q^{k(2\cl{n} - 2k+1) +k(k-1)} (q^{-2\cl{n}+1};q^2)_k \quad \text{resulting in} \\
 &  (q^{n - 2k +1};q)_{2k} = q^{k(2n - 2k +1)}  (q^{-2\fl{n}};q^2)_k (q^{-2\cl{n}+1};q^2)_k,
\end{align*}
where we have used the fact that 
\begin{equation}
\label{eq:fl-cl}
\cl{n} + \fl{n} = n.
\end{equation}
Now, apart from the summation-index independent term
$\frac{(q;q)_{2m}}{(q;q)_m^2}$, sum \eqref{eq:step-touch} gets transformed into
\begin{equation}
\sum_{k = 0}^{\fl{n}} q^{(2n+2m+1)k} \frac{(q^{-2\fl{n}};q^2)_k (q^{-2\cl{n}+1};q^2)_k}{ (q^2;q^2)_k (q^{2m+2};q^2)_{k}} = {}_2 \phi_1 \left(q^{-2\cl{n}+1},q^{-2\fl{n}};q^{2m+2};q^2;q^{2n+2m+1} \right)
\end{equation}
We can now use Identity \eqref{eq:Gauss1} with 
\[
N=\fl{n}, \quad a= q^{-2\cl{n}+1}, \quad c = q^{2m+2}
\]
as, using \eqref{eq:fl-cl}, one can see that
\[
\frac{c q^N}{a} = q^{2m + 2n +1}.
\]
So \eqref{eq:step-touch} becomes
\begin{equation}
\label{eq:step2-touch}
 \sum_{k=0}^{\fl{n}}  \frac{q^{2k(k+m)}}{(-q;q)_k(-q^{m+1};q)_{k}}
\qbinom{n}{2k}  T_{k,m}(q) = \frac{(q;q)_{2m}}{(q;q)_m^2}  \frac{(q^{2m + 2
\cl{n} +1};q^2)_{\fl{n}}}{(q^{2m+2};q^2)_{\fl{n}}}.
 \end{equation}
Next, we simplify the second ratio on the right-hand side:
 \begin{align}
 \label{eq:T1}
 & \frac{(q^{2m + 2 \cl{n} +1};q^2)_{\fl{n}}}{(q^{2m+2};q^2)_{\fl{n}}} = \frac{(q;q^2)_{n+m}}{(q;q^2)_{m+ \cl{n}}(q^{2m+2};q^2)_{\fl{n}}} \\
 \label{eq:T2}
 & = \frac{(q;q^2)_{n+m} (q^2;q^2)_{m + \fl{n}}}{(q^{2m+2};q^2)_{\fl{n}} (q;q^2)_{n+2m}} = \frac{(q;q^2)_{n+m} (q^2;q^2)_m}{(q;q)_{n+2m}} \\
 \label{eq:T3}
 & = \frac{(q;q)_{2(n+m)} (q^2;q^2)_m}{(q^2;q^2)_{n+m} (q;q)_{n+2m}} = \frac{(-q;q)_m (q;q)_m (q;q)_{2(n+m)}}{(-q;q)_{n+m}(q;q)_{n+m} (q;q)_{n+2m}} =
 \frac{(q;q)_m (q;q)_{2(n+m)}}{(-q^{m+1};q)_n (q;q)_{n+m} (q;q)_{n+2m}}
 \end{align}
 where 
\begin{itemize}
\item in \eqref{eq:T1}, we have used \eqref{eq:Kone} along with
\eqref{eq:fl-cl} to rewrite the numerator;
\item in \eqref{eq:T2}, we have used \eqref{eq:Kthree} to rewrite
$(q;q^2)_{m+ \cl{n}}$ and then \eqref{eq:Kone} to simplify the fraction;
\item in \eqref{eq:T3}, we have used \eqref{eq:Kthree} again to rewrite
$(q;q^2)_{n+m} $ and, finally, \eqref{eq:Ktwo} and \eqref{eq:Kone} to obtain
the factor $(-q^{m+1};q)_n$.
\end{itemize}

Returning to Equation \eqref{eq:step2-touch}, we have:
\begin{equation}
 \sum_{k=0}^{\fl{n}}  \frac{q^{2k(k+m)}}{(-q;q)_k(-q^{m+1};q)_{k}}  \qbinom{n}{2k}  T_{k,m}(q) = \frac{(q;q)_{2m}}{(q;q)_m^2}  \frac{(q;q)_m (q;q)_{2(n+m)}}{(-q^{m+1};q)_n (q;q)_{n+m} (q;q)_{n+2m}}
 \end{equation}
 or, cancelling common factors and bringing $ (-q^{m+1};q)_n $ to the left-hand side:
 \begin{equation}
  \sum_{k=0}^{\fl{n}} q^{2k(k+m)} \frac{(-q^{m+1};q)_{n-k}}{(-q;q)_k}  \qbinom{n}{2k}  T_{k,m}(q) = \frac{(q;q)_{2m}}{(q;q)_m}  \frac{ (q;q)_{2(n+m)}}{ (q;q)_{n+m} (q;q)_{n+2m}}
 \end{equation}
which is  (\ref{eq:q-Touchard}).
\end{proof}

Despite a similarity, Formula (\ref{eq:q-Touchard}) is different from the
one presented in \cite{Warnaar2011}:
\begin{equation}
T_{m,m+n}(q) = \sum_{k=0}^{\fl{n}} T_{m,k}(q) \sum_{j=k}^{n-k} q^{k(m+k) +j
(m+j)} \qbinom{n}{2k} \qbinom{n-2k}{j-k}.
\end{equation}

\subsection{A $q$-Koshy identity for the $q$-super Catalan Numbers}
\label{sec:Koshy}

The second result in \cite{Andrews2010} is a $q$-Koshy identity:
\begin{equation}
C_n(q)
= \sum_{r=1}^{[\frac{n+1}{2}]} (-1)^{r-1} q^{r(r-1)}
     \qbinom{n-r+1}{r} C_{n-r}(q) \frac{(-q^{n-r+1};q)_r}{(-q;q)_r}.
\end{equation}

Let us prove a generalization of this identity to all $q$-super Catalan
numbers:
\begin{state}
Let $n$ and $m$ be nonnegative integers with $m \leq n$. We have
\begin{equation}
\label{eq:Koshy}
T_{n,m}(q) =
 \sum_{r=1}^{\fl{n+\textcolor{red}{m}}}
     (-1)^{r-1} q^{r(r- \textcolor{red}{2m}+1)}
     \qbinom{n-r+\textcolor{red}{m}}{r} T_{n-r,m}(q)
     \frac{(-q^{n-r+1};q)_r}{(-q;q)_r},
\end{equation}
if $m>0$.

When $m=0$ the identity is
\begin{equation}
\label{eq:Koshy0}
\qbinom{2n}{n} =   \sum_{r=1}^{\fl{n}}
     (-1)^{r-1} q^{r(r+1)}
     \qbinom{n-r}{r} \qbinom{2n -2r}{n-r}
     \frac{(-q^{n-r+1};q)_r}{(-q;q)_r} + (-q;q)_n.
     \end{equation}
\end{state}

\begin{proof}
We will again be using the Gauss summation formulas 
($q$-Chu-Vandermonde) in the following version (II.6), \cite{GR}:
\begin{equation}
\label{eq:Gauss2}
{}_2 \phi_1(a,q^{-N};c;q;q) = \frac{(\frac{c}{a};q)_N}{(c;q)_N} a^N.
\end{equation}

Following the logic of the Andrews's proof, we move all terms of the case
$m>0$ to one side and consider
\begin{align}
& 
  \sum_{r=0}^{\fl{n+m}}
     (-1)^{r-1} q^{r(r- 2m+1)} \qbinom{n-r+m}{r} T_{n-r,m}(q)
     \frac{(-q^{n-r+1};q)_r}{(-q;q)_r} \\
& =
    \frac{\qnum{2m}!}{\qnum{m}!} \sum_{r=0}^{\fl{n+m}}
      (-1)^{r-1} q^{r(r- 2m+1)}  \frac{(-q^{n-r+1};q)_r}{(-q;q)_r}
      \frac1{\qnum{r}! \qnum{n - 2r +m}!}
      \frac{\qnum{2(n-r)}!}{\qnum{n-r}! } \\
\label{eq:Koshy1}
& =
   \frac{\qnum{2m}!}{\qnum{m}!} \sum_{r=0}^{\fl{n+m}}
      (-1)^{r-1} q^{r(r- 2m+1)}
      \frac{(-q^{n-r+1};q)_r (q;q)_{2(n-r)}}{(-q;q)_r (q;q)_r (q;q)_{n-2r+m} (q;q)_{n-r}} .
\end{align}
Now rewrite each term of the last sum as 
\begin{align}
\label{eq:Koshy-step1}
&  \frac{(-q^{n-r+1};q)_r (q;q)_{2(n-r)}}{(q^2;q^2)_r (q;q)_{n-2r+m} (q;q)_{n-r}}  =  \frac{(-q^{n-r+1};q)_r (q^{n-r+1};q)_r (q^{n-2r+m+1};q)_{n-m}}{(q^2;q^2)_r (q;q)_{n-r}(q^{n-r+1};q)_r} \\
\label{eq:Koshy-step2}
& = \frac{ (q^{2n-2r+2};q^2)_r (q^{n-2r+m+1};q)_{n-m}}{(q^2;q^2)_r (q;q)_{n}} = \frac1{(q;q)_n} \frac{ (q^{2n-2r+2};q^2)_r (q^{n-2r+m+1};q)_{n-m}(q^{n-2r+1};q^2)_r}{(q^2;q^2)_r (q^{n-2r+1};q^2)_r} \\
\label{eq:Koshy-step3}
& = \frac{ (q^{n+m+1};q)_{n-m}}{(q;q)_n}
\frac{(q^{n-2r+m+1};q)_{2r}}{(q^2;q^2)_r (q^{n-2r+1};q^2)_r} = (-1)^r q^{r(2m
- r + 1)} \frac{ (q^{n+m+1};q)_{n-m}}{(q;q)_n}  \frac{(q^{-(n+m)+1};q^2)_r
(q^{-(n+m)};q^2)_r}{(q^2;q^2)_r (q^{-(2n-1)};q^2)_r},
 \end{align}
where
\begin{itemize}
\item in \eqref{eq:Koshy-step1}, we multiplied both the numerator and the
denominator by $( q^{ n- r +1} ; q )_r $.  We also applied \eqref{eq:Kone} to
simplify
\[
      \frac{( q ; q )_{2( n- r )}} {( q ; q )_{n -2 r + m}} = ( q^{n - 2 r + m +1} ; q )_{m -n} 
 \]
\item in \eqref{eq:Koshy-step2}, in the numerator we have created even powers
of $q$ running from $q^{2n - 2r +2}$ up to $q^{2n-2}$ by combining the factor
$( q^{n -r +1} ; q )_r $ with  $( -q^{n - r +1} ; q )_r $  and using
\eqref{eq:Ktwo}.
We have also created matching odd powers by multiplying both the numerator and
denominator by $(q^{n-2r+1};q^2)_r$ as well as combining the factor $( q^{n -
r +1} ; q )_r$ and the factor $( q ; q )_{n -r} $ in the denominator to give
the factor $( q ; q )_n$  by using \eqref{eq:Kone}.
\item in \eqref{eq:Koshy-step3}, we have first split the product in the
numerator in two and then, in the second equality split the product in the
numerator into odd and even powers. Finally we have reversed the order of
multiplication both in the numerator and denominator by using \eqref{eq:Kfour}.
\end{itemize}

Ignoring the overall factor $ \frac{( q^{m +1} ; q )_m (q^{n+m +1};q)_{n-m}}{
(q;q)_n}$, the sum on the right-hand side of (\ref{eq:Koshy1}) gets
transformed into
\begin{equation}
\label{eq:Koshy-final1}
\sum_{r \geq 0} q^{2r} \frac{(q^{-(n+m) +1};q^2)_r(q^{-(n+m )};q^2)_r
}{(q^2;q^2)_r (q^{-(2n-1)};q^2)_r} =  {}_2 \phi_1(a,b,q^{-(2 n-1)}, q^2;q^{2})
\end{equation}
with  $\{ a; b \} = \{ q^{-(n+m) }; q^{-(n + m) +1} \}$. The finiteness of the
sum in the $ {}_2 \phi_1 $ follows from the fact that when $r > \fl{n + m}$ ,
either $( q^a ; q^2 )_r $ or $( q^b ; q^2 )_r $ vanishes.\\

Now we apply Equation~\eqref{eq:Gauss2} with $N=\fl{m+n}$ and the right-hand
side in~\eqref{eq:Koshy-final1} becomes
\begin{equation}
{}_2 \phi_1( a, b, q^{-(2 n-1)}, q^2 ; q^{2}) =
\frac{(q^{-2\cl{n-m}};q^2)_N}{(q^{-(2n-1)};q^2)_N} q^{(1-n-m)N}.
\end{equation}

Finally, if $m>0$, the term $(q^{-2\cl{n-m}};q^2)_N$  is zero, hence
proving~\eqref{eq:Koshy}. \\
If $m=0$, the right-hand side simplifies into $-(q;q)_n$ after a simple
expansion, hence proving~\eqref{eq:Koshy0}.
\end{proof}

\subsection{A $q$-Reed Dawson identity for $q$-super Catalan numbers}
\label{sec-reed_dawson}

Andrews in~\cite{Andrews1974}, Theorem 5.4, establishes a $q$-analog of a Reed
Dawson identity for central binomial coefficients, the $T_{n,0}(q)$ in our
notation:
\begin{equation}
 \sum_{k=0}^n
    (-1)^k  q^{\frac{k(k-1)}{2}} \frac{1}{(-q;q)_{n-k}}\qbinom{n}{k}
    \qbinom{2n-2k}{n-k}
= 
\begin{cases}
\frac{q^{\frac{n^2}{2}}}{(-q;q)_{n/2}^2} \qbinom{n}{n/2},
    & \text{if $n$ is even,} \\
 0, & \text{if $n$ is odd.}
\end{cases} 
\end{equation}
We generalize it to all $q$-super Catalan numbers:
\begin{state}
\label{eq:qRD}
Let $m$ and $n$ be nonnegative integers. We have the following identity that
generalizes the $q$-Reed Dawson identity to all super Catalan numbers:
\begin{align}
& \sum_{k=0}^{n-\textcolor{red}{m}}  (-1)^k  q^{\frac{k(k-1)}{2}}
     \frac{1}{(-q;q)_{n-k}}\qbinom{n-\textcolor{red}{m}}{k} T_{n-k, m}(q) \\
&= \begin{cases}
  \frac{q^{(n^2 -\textcolor{red}{m^2})/2} }{(-q;q)_{(n-\textcolor{red}{m})/2} (-q;q)_{(n+\textcolor{red}{m})/2} }
  \qbinom{\textcolor{red}{2m}}{\textcolor{red}{m}}
  \qbinom{n}{\textcolor{red}{m}}^{-1}
  \qbinom{n}{(n-\textcolor{red}{m})/2}, & \text{if $ n -m $ is even,} \\
0, & \text{otherwise.}
\end{cases}
\end{align}
\end{state}

\begin{proof}
We transform a generic term in the sum of Proposition~\ref{eq:qRD} as:
\begin{align}
&  \frac{1}{(-q;q)_{n-k}}\qbinom{n-m}{k} T_{n-k, m}(q)
 = \frac{1}{(-q;q)_{n-k}}\qbinom{n-m}{k}
   \frac{\qnum{2m}! \qnum{2n-2k}!}{\qnum{m}! \qnum{n-k}! \qnum{n -k +m}!} \\
\label{eq:qRD1}
& = \frac{(q;q)_{2m}}{(q;q)_m} \frac{(q^{n-m-k+1};q)_k}{(q;q)_k}
   \frac{(q;q)_{2n-2k}}{(q^2;q^2)_{n-k}(q;q)_{n-k+m}} \\
\label{eq:qRD2}
& = (-1)^k q^{\frac{k(2n - 2m -k +1)}{2}} \frac{(q;q)_{2m}}{(q;q)_m}
    \frac{(q^{-(n-m)};q)_k  (q;q^2)_{n - k}(q^{n + m -k +1};q)_k}{(q;q)_k (q;q)_{n-k+m} (q^{n + m -k +1};q)_k} \\
\label{eq:qRD3}
& = q^{\frac{k(2n - 2m -k +1)}{2}}  q^{\frac{k(2n + 2m - k +1)}{2}}
   \frac{(q;q)_{2m}}{(q;q)_m}
    \frac{(q^{-(n-m)};q)_k (q^{-(n+m)};q)_k  (q;q^2)_{n - k} (q^{2n - 2k +1};q^2)_{k} }{(q;q)_k (q;q)_{n+m} (q^{2n - 2k +1};q^2)_{k}} \\
\label{eq:qRD4}
& = (-1)^k q^{k(k-2n+1)} q^{-k(k - 2n )} \frac{(q;q)_{2m}}{(q;q)_m}
    \frac{(q^{-(n-m)};q)_k (q^{-(n+m)};q)_k (q;q^2)_n}{(q;q)_k (q;q)_{n+m}
                                                       (q^{-2n+1};q^2)_k}.
\end{align}
where 
\begin{itemize}
\item in the denominator of \eqref{eq:qRD1}, we have combined $ (-q;q)_{n-k}
(q;q)_{n-k} $ to get $(q^2;q^2)_{n-k}$ using \eqref{eq:Kone}.
\item in \eqref{eq:qRD2}, we used \eqref{eq:Ktwo} to write $( q ; q^2 )_{n-k}
= \frac{( q ; q )_{2( n - k )}}{ ( q^2 ; q^2 )_{n -k }}$, then multiplied the
numerator and the denominator by $( q^{n + m - k +1} ; q )_k$, and used
\eqref{eq:Kfour} to reverse the powers in $ ( q^{n - m -k +1} ; q )_k$.
\item in \eqref{eq:qRD3}, we multiplied the numerator and the denominator by
$( q^{2 n-2 k +1} ; q^2 )_k $, reversed the powers in
$ ( q^{n + m -k +1} ; q)_k $ in the numerator using \eqref{eq:Kfour},
used \eqref{eq:Kone} to write
$( q ; q )_{n + m -k} ( q^{n + m -k +1} ; q )_k = ( q ; q )_{n + m} $ in the
denominator.
\item in \eqref{eq:qRD4}, we used \eqref{eq:Kone} to write
$( q ; q^2 )_{n -k }( q^{2( n -k )+1} ; q^2 )_k = ( q ; q^2 )_n $ in the
numerator and \eqref{eq:Kfour} to reverse powers in
$( q^{2 n -2 k +1} ; q^2)_k$ in the denominator.
\end{itemize}

Summing over $k$ then leads to
\begin{align*}
& \sum \limits_{k=0}^{n-m}  (-1)^k  q^{\frac{k(k-1)}{2}}
  \frac{1}{(-q;q)_{n-k}}\qbinom{n-m}{k} T_{n-k, m}(q) \\
&
=  \frac{(q;q)_{2m}}{(q;q)_m} \frac{(q;q^2)_n}{ (q;q)_{n+m}}
   \sum_{k=0}^{n-m}
    \frac{(q^{-(n-m)};q)_k (q^{-(n+m)};q)_k }{(q;q)_k (q^{-2n+1};q^2)_k}
    q^{\frac{k(k+1)}{2}}.
\end{align*}
We now apply the $q$-analog of the second theorem of
Gauss~\cite{Andrews1974}, \cite{Johnson}~(Theorem 50):
\begin{equation}
\label{eq:second_Gauss}
\sum_{k=0}^{\infty}
  \frac{(a;q)_k (b;q)_k }{(q;q)_k (c;q^2)_k} q^{\frac{k(k+1)}{2}}
= \frac{(aq;q^2)_{\infty}(bq;q^2)_{\infty}}{(q;q^2)_{\infty} (a b q;q^2)_{\infty}}
\quad \text{if} \quad c = abq,
\end{equation}
with
\begin{align*}
& a = q^{-(n-m)}  \quad b = q^{-(n+m)} 
\quad c = q^{-2n +1},
\end{align*}
so that
\begin{align*}
& \sum_{k=0}^{n-m}
  (-1)^k  q^{\frac{k(k-1)}{2}}
  \frac{1}{(-q;q)_{n-k}}\qbinom{n-m}{k} T_{n-k, m}(q) \\
& = \frac{(q;q)_{2m}}{(q;q)_m} \frac{(q;q^2)_n}{ (q;q)_{n+m}}
   \frac{(q^{-(n-m) +1};q^2)_{\infty} (q^{-(n+m) +1};q^2)_{\infty} }{ (q;q^2)_{\infty} (q^{-2n+1};q^2)_{\infty}} \\
& = 
\frac{(q;q)_{2m}}{(q;q)_m} \frac{(q;q^2)_n}{ (q;q)_{n+m}}
 \frac{ (q^{-(n-m) +1};q^2)_{(n-m)/2} (q^{-(n+m) +1};q^2)_{(n+m)/2} }{(q^{-2n+1};q^2)_{n}}, \quad \text{if $n-m$ is even} \\
\end{align*}
Here we have used \eqref{eq:Kone} to write
\begin{align*}
& ( q^{-( n -m )+1} ; q^2 )_{\infty} = ( q^{-( n -m )+1} ; q^2 )_{\frac{(n-m)}{2}} ( q ; q^2 )_{\infty} \\
& ( q^{-( n + m )+1} ; q^2 )_{\infty} = ( q^{-( n + m )+1} ; q^2 )_{\frac{( n+m )}{2}} ( q ; q^2 )_{\infty} \\
& ( q^{-2 n +1} ; q^2 )_{\infty} = ( q^{-2 n +1} ; q^2 )_n ( q ; q^2 )_{\infty}
\end{align*}
and canceled equal infinite products. \\
The product in the numerator vanishes when $n-m$ is odd, hence the statement.
\end{proof}

\section{Convolution formulas}
\label{sec:convo}

In the classical case at $q=1$, convolution formulas are already known among
central binomial coefficients, Catalan numbers, and between both sequences.
We present here $q$-versions of these convolution formulas.

When $q=1$, the cases $m=0$ and $m=1$ are respectively known as Sved and
Segner formulas. In our notation, they read

\begin{enumerate}
\item{central binomial coefficients, aka Sved formula, \cite{Koshy}  p. 99,
\cite{Gould} (3.90):}
\begin{equation}
\label{eq:cbc}
\sum_{k=0}^n \binom{2 k}{k} \binom{2n - 2k}{n-k} = 2^{2n} 
\quad \Leftrightarrow \quad
\sum_{k=0}^n T_{k,0} T_{n-k,0} = 2^{2n} .
\end{equation}
\item{Catalan convolution, aka Segner relation, \cite{Koshy} p. 114:} 
\begin{equation}
\label{eq:Ca-co}
\sum_{k=0}^{n} C_{n-k} C_k = C_{n+1}
\quad \Leftrightarrow \quad
\sum_{k=0}^{n}T_{n-k,1} T_{k,1} = 2 T_{n+1,1} .
\end{equation}
\end{enumerate}

Both of these equations generalize to the $q$ case.

Our main tool in this Section is the $\frac1{2}$-trick developed
in~\cite{Gessel} that relates $q$-super Catalan numbers to $q$-binomial
coefficients. The following formula for any $q$-super Catalan can be established by direct
computation:

\begin{lem}
Let $k$ and $m$ be nonnegative integers. We have
\begin{equation}
\label{eq:half}
 \qbinomm{m -\frac1{2}}{k +m}
  = \frac{(-1)^k q^{-k^2} }{(-q;q)_{m} (-q;q)_{k} (-q;q)_{k +m} } T_{k,m}(q).
\end{equation}
\end{lem}
\begin{proof}
By the definition of the $q$-binomial (\ref{eq:binom})
\begin{align}
\label{eq:b1}
& \qbinomm{m -\frac1{2}}{k +m} =\frac{(q^{1-2k};q^2)_{k+m}}{(q^2;q^2)_{k+m}} =
\frac{(q;q^2)_m (q^{-2k+1};q^2)_k}{(q^2;q^2)_{k+m}}   \\
\label{eq:b2}
& = (-1)^k q^{-k^2} \frac{  ( q ; q^2 )_m ( q ; q^2 )_k }{(q^2;q^2)_{k+m}} =
(-1)^k q^{-k^2} \frac{ ( q ; q )_{2 m}( q ; q )_{2 k} }{ ( q^2 ; q^2 )_m  (q^2
; q^2 )_k(q^2;q^2)_{k+m}} \\
\label{eq:b3}
& =  \frac{(-1)^k q^{-k^2} }{(-q;q)_{m} (-q;q)_{k} (-q;q)_{k +m} } \frac{ ( q
; q )_{2 m}( q ; q )_{2 k} }{ ( q ; q)_m  (q ; q )_k (q;q)_{k+m}}  =
\frac{(-1)^k q^{-k^2} }{(-q;q)_{m} (-q;q)_{k} (-q;q)_{k +m} } T_{k,m}(q),
\end{align}
where 
\begin{itemize}
\item in \eqref{eq:b1}, we have used \eqref{eq:Kone} to write  $( q^{1 - 2 k} ; q^2 )_{k + m} = ( q ; q^2 )_m ( q^{2 k +1} ; q^2 )_k $;
\item in \eqref{eq:b2}, we have first reversed powers as in \eqref{eq:Kfour}
to write $( q^{2 k +1} ; q^2 )_k = ( -1)^k q^{k^2} ( q ; q^2 )_k $ and then $(
q ; q^2 )_m = \frac{( q ; q )_{2 m}}{ ( q^2 ; q^2 )_m } $ (and the same for $(
q ; q^2 )_k$ ) using \eqref{eq:Kone};
\item in \eqref{eq:b3}, we have used \eqref{eq:Ktwo} to rewrite $( q^2 ; q^2
)_m = ( -q ; q )_m ( q ; q )_m $ and identically for $( q^2 ; q^2 )_k$ and
$( q^2 ; q^2 )_{k+m}$.
\end{itemize}
~\end{proof}

Let us return to convolution formulas and  consider the $m=0$ case first,
\emph{i.e.} the $q$-Sved identity.

\begin{state}
Let $n$ be a nonnegative integer. We have
\begin{equation}
\label{eq-q-Sved}
\sum_{k=0}^n \frac{q^{n-k}}{(-q;q)_k^2 (-q;q)_{n-k}^2} \qbinom{2k}{k}
\qbinom{2n-2k}{n-k} = 1.
\end{equation}
\end{state}

\begin{proof}
We use the $q$-Vandermonde identity and identify each of the $q$-binomials with the corresponding central binomial coefficient as in~\eqref{eq:half} :
\begin{align*}
&  (-1)^n q^{- n(n-1)} =  \qbinomm{-1}{n} = \sum_{k=0}^{n} q^{2 k (-n
-\frac1{2}  + k  )} \qbinomm{ -\frac1{2}}{k} \qbinomm{ -\frac1{2}}{n  - k }
\\
 & =\sum_{k=0}^{n} q^{-2k n -k +2k^2} (-1)^k \frac{q^{-k^2}}{(-q;q)_k^2}
\qbinom{2k}{k}   (-1)^{n-k} \frac{q^{-(n-k)^2}}{(-q;q)_{n-k}^2}
\qbinom{2n-2k}{n-k} \\
 & = \sum_{k=0}^{n} (-1)^n q^{-n^2-k} \frac1{(-q;q)_k^2 (-q;q)_{n-k}^2} \qbinom{2k}{k}  \qbinom{2n-2k}{n-k}.
\end{align*}
Simplifying, we get~\eqref{eq-q-Sved}.
\end{proof}

Next we consider $m=1$ and obtain the $q$-Segner identity:

\begin{state}
Let $n$ be a nonnegative integer. We have
\begin{equation}
\label{eq:q-Catalan-conv}
\sum_{k=0}^n
  q^k
  \frac{(-q^{k+1} ;q)_{n-k +1}(-q^{k+2};q)_{n-k} }{ (-q;q)_{n - k+1} (-q ;q)_{n-k}}
  T_{k,1}(q) T_{n-k,1}(q)
=  (1+q) T_{n+1,1}(q).
\end{equation}
\end{state}
\begin{proof}
We shall again use the $q$-Vandermonde formula and identify each of the
$q$-binomials with the corresponding $q$-Catalan as in~\eqref{eq:half} and
then isolate the boundary terms, \emph{i.e.}, $k=-1$ and $k = n+1$:
\begin{align}
& \sum_{k=-1}^{n+1} q^{2 (k+1) (-n -\frac1{2}  + k  )}\qbinomm{\frac1{2}}{k +1} \qbinomm{\frac1{2}}{n  - k + 1} = \qbinomm{1}{n+2} = 0 \\
 & 0 =  \qbinomm{\frac1{2}}{(n  +1) + 1}  + q^{n+2} \qbinomm{\frac1{2}}{(n+1) +1} \\
 & + \sum_{k=0}^{n} q^{2 (k +1) ( -\frac1{2}  - n  +k  )}  \frac{(-1)^k q^{-k^2} }{(-q;q)_{1} (-q;q)_{k} (-q;q)_{k +1} } T_{k,1}(q)
 \frac{(-1)^{n-k} q^{-(n-k)^2} }{(-q;q)_{1} (-q;q)_{n-k} (-q;q)_{n-k +1} } T_{n-k,1}(q) \notag \\
 & 0 = \left(1 + q^{n+2} \right) \frac{(-1)^{n+1} q^{-(n+1)^2} }{(-q;q)_{1} (-q;q)_{n+1} (-q;q)_{n +2} } T_{n+1,1}(q) \\
& + \sum_{k=0}^{n}  (-1)^n \frac{q^{k - (n+1)^2} }{(-q;q)_{1}^2 (-q;q)_{k} (-q;q)_{k +1} (-q;q)_{n-k} (-q;q)_{n-k +1}} T_{k,1}(q)T_{n-k,1}(q) \notag \\
&  \frac1{ (-q ;q)_{n+1}^2} T_{n+1,1}(q) = \sum_{k=0}^n \frac{q^k  }{(-q;q)_1(-q;q)_{k +1} (-q ;q)_k (-q;q)_{n - k+1} (-q ;q)_{n-k}} T_{k,1}(q) T_{n-k,1}(q) \\
& (1 + q)T_{n+1,1}(q)  = \sum_{k=0}^n q^k \frac{(-q^{k+1} ;q)_{n-k +1}(-q^{k+2};q)_{n-k} }{ (-q;q)_{n - k+1} (-q ;q)_{n-k}} T_{k,1}(q) T_{n-k,1}(q).
\end{align}
\end{proof}

As mentioned before, there also exist convolution formulas involving both
central binomial coefficients and Catalan numbers,
see~\cite{Koshy} (5.24), p.~140:
\begin{equation}
\label{eq:cbc-cat}
\sum_{k=0}^n 2 C_k \binom{2n-2k}{n-k} = \binom{2n + 2}{n+1} 
\quad  \Leftrightarrow \quad
\sum_{k=0}^n T_{k,1} T_{n-k,0} = T_{n+1,0}.
\end{equation}

We have the following generalization of this formula:

\begin{state} 
Let $n$ be a nonnegative integer. We have
\begin{equation}
\label{eq:q-cbc-cat}
\sum_{k=0}^n \frac{q^{n-k }(-q;q)_{n+1}^2}{(1+q)(- q;q)_k^2 (-q;q)_{n-k}
(-q;q)_{n-k+1}} T_{k,0}(q) T_{n-k,1}(q) = T_{n+1,0}(q).
\end{equation}
\end{state}

\begin{proof}
The proof proceeds by following exactly the same steps as in the proof of
Equation~\eqref{eq:q-Catalan-conv}.
\end{proof}

Using the same approach, it is possible to derive a general convolution
formula between any two $q$-super Catalan numbers $T_{k,m}(q)$ and
$T_{n-k,l}(q)$. We will detail this formula and closely related formulas,
and explore their implications in a forthcoming paper.

\section{Generating function for $q$-Catalan numbers}
\label{sec-genf}

In this Section we will derive the generating function for the $q$-Catalan
numbers from the convolution formulas. Our strategy is to derive a generating
function for the $q$-central binomial numbers first and then use the
convolution formula~\eqref{eq:q-cbc-cat} to derive a formula for
the generating function of the $q$-Catalan numbers.

Recall from Section~\ref{sec:convo}, the $q$-Sved formula \eqref{eq-q-Sved}:
\[
\sum_{k=0}^n \frac{q^{n-k}}{(-q;q)_k^2 (-q;q)_{n-k}^2} \qbinom{2k}{k}
 \qbinom{2n-2k}{n-k} = 1.
\]
We shall translate this identity into a functional equation on a generating
function for $q$-central binomial coefficients that will play a central
r\^ole in what comes next.

We begin by defining the generating function for the $q$-super Catalan numbers:

\begin{definition}
Let $m$ be a nonnegative integer.
The generating function of all $T_{n,m}(q)$ is
\begin{equation}
\label{eq:fm-def}
f_m(q,t) =
  \frac1{(-q;q)_m}\sum_{n=0}^{\infty}
  \frac{T_{n,m}(q)}{ (-q;q)_n (-q;q)_{n+m}}  t^n.
\end{equation}
\end{definition}

With these definitions, Equation~\eqref{eq-q-Sved} directly rewrites as the
functional equation
\begin{equation}
\label{eq:f0-func}
f_0(q,t) f_0(q,qt) = \frac1{1-t}.
\end{equation}

From this formula and the condition $f_0(q,0)=1$, one can obtain a product
formula for $f_0(q,t)$:

\begin{state} 
The generating function for the $q$-central binomial coefficients is
\begin{equation}
\label{eq:f0}
f_0(q,t) = \frac{(1- q t) (1 - q^3 t) \ldots}{(1- t)(1 - q^2 t) \ldots }
 = \frac{(qt;q^2)_{\infty}}{(t;q^2)_{\infty}}.
\end{equation}
\end{state}

\begin{proof}
Let us provide two proofs of this formula.
First, note that $f_0(q,t)$ of Equation~\eqref{eq:fm-def} and the product on
the right-hand side of Equation~\eqref{eq:f0} are polynomials in $t$ with
constant term $1$ and both satisfy the functional
equation~\eqref{eq:f0-func}). Indeed,
\begin{equation}
 f_0(q,t) f_0(q, q t) = \frac{(1- q t) (1 - q^3 t) \ldots}{(1- t)(1 - q^2 t)
\ldots } \times \frac{(1- q^2 t) (1 - q^4 t) \ldots}{(1- q t)(1 - q^3 t)
\ldots } = \frac1{1-t}.
\end{equation}

Formula~\eqref{eq:f0} can also be proven directly by comparing the
coefficients in the sum and product expressions.
\begin{enumerate}
\item Rewrite the sum
\begin{align*}
& f_0(q,t)
 = \sum_{k=0}^{\infty} \frac1{(-q;q)_k^2 } \qbinom{2k}{k} t^k
 = \sum_{k=0}^{\infty} \frac1{(-q;q)_k^2 }
     \frac{(q;q)_{2k}}{(q;q)_k^2} t^k  = \sum_{k=0}^{\infty}
      \frac{(q;q)_{2k}}{(q^2;q^2)_k^2}t^k
  =  \sum_{k=0}^{\infty} \frac{(q;q^2)_k}{(q^2;q^2)_k} t^k,
\end{align*}
where we have first used \eqref{eq:Kone} and then \eqref{eq:Kthree} to simplify the expression.
\item Extract the coefficient from the product.
Appeal to the result of Exercise 5, Ch.2 of~\cite{Mac}:
\[
H(t) = \prod_{i=0}^{\infty} \frac{1 - b q^i t}{1 - a q^i t} = \sum_{k \geq 0} \prod_{i=1}^k \frac{a - b q^{i-1}}{1 - q^i} t^k
\]
with 
\[
q \rightarrow q^2, \quad a =1, \quad b = q
\]
so that
\[
 \frac{(qt;q^2)_{\infty}}{(t;q^2)_{\infty}} = \sum_{k \geq 0} \prod_{i=1}^k
\frac{1 - q q^{2i-2}}{1 - q^{2i}} t^k =  \sum_{k \geq 0}
\frac{(q;q^2)_k}{(q^2;q^2)_k} t^k.
\]
\end{enumerate}
\end{proof}

The generating function for the $q$-Catalan numbers is particularly simple.
In order to derive it, recall the convolution of $T_{n,0}(q)$ and
$T_{n,1}(q)$, Equation (\ref{eq:q-cbc-cat}):
\[
\sum_{ k =0}^n \frac{T_{k, 0} ( q )}{ ( -q ; q )_k ^2} \cdot \frac{q^{n-k}}{1
+ q} \frac{T_{n -k ,1}(q)}{ ( -q ; q )_{n -k }( -q ; q )_{n -k +1}}  =
\frac1{(-q;q)_{n+1}^2} T_{n+1,0}(q).
\]

This $q$-convolution formula is equivalent to the following functional
equation in the generating series $f_0$ and $f_1$:
\begin{align*}
& f_0(q,t) f_1(q,q t) = \frac{f_0(q,t)-1}{t}. \\
\end{align*}
Solving this equation for $f_1(q,t)$, we find

\begin{state}
The generating function for the $q$-Catalan numbers satisfies
\begin{equation}
\label{eq:f1}
f_1(q,t) =
  \frac{q \left(f_0\left(q,\frac{t}{q} \right) - 1 \right)}{t
f_0\left(q,\frac{t}{q} \right)},
\end{equation}
where $f_0(q,t)$ is as in Equation~\eqref{eq:f0}.
\end{state}

Note that putting $q=1$ in the previous formula allows one to directly
recover the usual formula for the generating series of the classical Catalan
numbers.

\section{$\gamma$-positivity and Narayana refinements}
\label{sec:gamma}

The notion of $\gamma$-positivity~\cite{FS} is a property of a polynomial
that ensures that it is both unimodal and palindromic.
It requires that the polynomial can be expressed as
\begin{equation}
\label{eqn:eq1}
f(t)=\sum_{k=0}^{\lfloor n/2 \rfloor} \gamma_k t^k(1+t)^{n-2k}
\end{equation}
with nonnegative real values
$\gamma_0$,$\gamma_1$,...,$\gamma_{\lfloor n/2 \rfloor}$.

For instance, Simion and Ullman in~\cite{Simion} prove that the type A
Narayana polynomials
\begin{equation}
\label{eq:Npoly}
\mathcal{N}_n(t) = \sum_{k=0}^{n-1} N_{n,k} t^k
\end{equation}
built from the Narayana numbers
$N_{n,k} = \frac{1}{n} \binom{n}{k+1}\binom{n}{k}$ that refine Catalan
numbers satisfy
\begin{equation}
\label{eq:SU}
\mathcal{N}_n(t) = \sum_{k=0}^{\lfloor\frac{n-1}{2}\rfloor}
            \binom{n-1}{2k} C_k t^k (1 + t)^{n-1-2k}.
\end{equation}

This is case $m=1$ for us. The super Catalan numbers can be directly
refined in a similar way:

\begin{equation}
T_{n,m}(1)
 =  \binom{2m}{m}\binom{n}{m}^{-1} \binom{2n}{n-m}
 =  \binom{2m}{m}\binom{n}{m}^{-1}
    \sum_{k=0}^{n-m} \binom{n}{k} \binom{n}{n-k-m}.
\end{equation}

Accordingly, the $m$-Narayana numbers and their generating functions, the
$m$-Narayana polynomials are

\begin{definition}[\cite{Wachs}]
Let $m$, $n$, and $k$ be three nonnegative integers, such that $m\leq n$
and $k\leq n-m$.
The $m$-Narayana numbers are
\begin{equation}
N^{(m)}_{n,k} = \binom{2m}{m}\binom{n}{m}^{-1} \binom{n}{k} \binom{n}{k+m}.
\end{equation}
The generating function of the $m$-Narayana numbers, the $m$-Narayana
polynomial, is
\begin{equation}
\label{def:mNara}
\mathcal{N}^{(m)}_{n}(1,t) = \sum_{k=0}^{n-m} N^{(m)}_{n,k} t^k.
\end{equation}
\end{definition}

This is the $q=1$ version of the definition that appears in a talk by Wachs (see~\cite{Wachs}) and studied in an unpublished paper by Krattenthaler
and Wachs \cite{K}.
Note that putting $m=1$ gives back the classical Narayana numbers (up
to a factor $2$ as the super Catalan numbers give back the Catalan numbers at
$m=1$ up to the same factor). The same is true for the generating series:
$\mathcal{N}^{(1)}_{n}(1,t)= 2\mathcal{N}_n(t)$.

We then observe that Equation~\eqref{eq:SU} can be generalized to all $m$:

\begin{state}
\label{gamma}
Let $m$ and $n$ be nonnegative integers with $m\leq n$.
The polynomial $\mathcal{N}^{(m)}_{n}(1,t)$ is $\gamma$-positive since it
satisfies
\begin{equation}
\label{eq:mSU}
\mathcal{N}^{(m)}_{n}(1,t) 
 = \sum_{k=0}^{\lfloor\frac{n-\textcolor{red}{m}}{2}\rfloor}
     \binom{n-\textcolor{red}{m}}{2k} T_{k,m}\,
     t^k (1+t)^{n-\textcolor{red}{m}-2k}.
\end{equation}
\end{state}

While this statement is a corollary of a theorem in the presentation of Wachs,
we shall prove it in our own way.

\begin{proof}
Comparing the coefficients of $t^k$ on both sides of (\ref{eq:mSU}), the
statement is equivalent to the identity
\begin{equation}
\label{eq:gamma-id}
N^{(m)}_{n,k} =  \sum_{s=0}^{k}\binom{n-m}{2s} T_{m,s} \binom{n-m-2s}{k-s}.
\end{equation}

The validity of this identity will be established as a corollary of our
$q$-version of this identity in Section~\ref{sec-Na_pos},
Equation~\eqref{eq:Na-pos}.

Thus, we have that $\mathcal{N}^{(m)}_{n}(t)$ is $\gamma$-positive for all $m$.
\end{proof}

In~\cite{BlancoPetersen}, Blanco and Petersen give a combinatorial proof
in our case $m=0$ of Equation~\eqref{eq:mSU} as a symmetric boolean
decomposition of the lattice $NC^B(n)$ of type B noncrossing partitions.
That proves the $\gamma$-positivity of type B Narayana polynomials. They also
mention that it would be interesting to have such a decomposition in the type
D case. We shall provide a proof that relies on a different decomposition of
$NC^B(n)$ that has the benefit of being compatible with type D noncrossing
partitions. Note that type D is unrelated to the super Catalan family.

\subsection{Type B noncrossing partitions}

Simion and Ullman~\cite{Simion} proved $\gamma$-positivity in their study of
the non-crossing partition lattice. In~\cite{Reiner_1997}, Reiner introduced
an analog of noncrossing partitions for type B with the following definition.

\begin{definition}
A {type B partition} is a partition $\pi$
of $[\pm n]=\{+1,\dots,+n,-1,\dots,-n\}$ such that if
$\{i_1,i_2,\dots,i_k\}$ is a block of $\pi$ then $\{-i_1,-i_2,\dots,-i_k\}$ is
also a block of $\pi$ and there is at most one block containing both $i$ and
$-i$ for some $i$. If such a block exists, it is called the
{zero-block} of $\pi$.
\end{definition}

The natural representation of a type B partition (Figure \ref{bpart}) is
obtained by placing the numbers $1,2,\dots,n,-1,-2,\dots,-n$ clockwise on a
circle and drawing chords between successive numbers that are in the same
block. If none of the chords cross, then the type B partition is said to be
noncrossing and belongs to the set $NC^B(n)$ of all noncrossing partitions of
type B. Just as in type A, the type B noncrossing partitions form a lattice
under the refinement order. This lattice is ranked by the function
$r(\pi)=n-nbl(\pi)/2$ where $nbl(\pi)$ is the number of nonzero blocks
in~$\pi$.

\def\f {2}
\begin{figure}[ht]
\begin{center}
\begin{tikzpicture}[scale=1]
\draw (0,0) circle (\f cm);
\foreach \a in {1,2,3}
		\pgfmathsetmacro\x{int(4-\a))}
        \filldraw ({\f*cos(30*\a)}, {\f*sin(30*\a)}) circle (0.09cm) node[above right](tr){\x};
\foreach \a in {-2,-1,0}
		\pgfmathsetmacro\x{int(4-\a))}
        \filldraw ({\f*cos(30*\a)}, {\f*sin(30*\a)}) circle (0.09cm) node[below right](sr){\x};
\foreach \a in {4,5,6}
		\pgfmathsetmacro\y{int(\a-10))}
        \filldraw ({\f*cos(30*\a)}, {\f*sin(30*\a)}) circle (0.09cm) node[above left] (ur){\y};
\foreach \a in {7,8,9}
		\pgfmathsetmacro\y{int(\a-10))}
        \filldraw ({\f*cos(30*\a)}, {\f*sin(30*\a)}) circle (0.09cm) node[below left](vr){\y};
\def\i{3}
\def\j{7}
\draw ({\f*cos(30*\i)}, {\f*sin(30*\i)})--({\f*cos(30*\j)}, {\f*sin(30*\j)});
\draw ({\f*cos(30*(\i+6))}, {\f*sin(30*(\i+6))})--({\f*cos(30*(\j+6))}, {\f*sin(30*(\j+6))});
\def\j{4}
\draw ({\f*cos(30*\i)}, {\f*sin(30*\i)})--({\f*cos(30*\j)}, {\f*sin(30*\j)});
\draw ({\f*cos(30*(\i+6))}, {\f*sin(30*(\i+6))})--({\f*cos(30*(\j+6))}, {\f*sin(30*(\j+6))});
\def\i{7}
\draw ({\f*cos(30*\i)}, {\f*sin(30*\i)})--({\f*cos(30*\j)}, {\f*sin(30*\j)});
\draw ({\f*cos(30*(\i+6))}, {\f*sin(30*(\i+6))})--({\f*cos(30*(\j+6))}, {\f*sin(30*(\j+6))});
\def\i{2}
\def\j{8}
\draw ({\f*cos(30*\i)}, {\f*sin(30*\i)})--({\f*cos(30*\j)}, {\f*sin(30*\j)});
\end{tikzpicture}
\end{center}
\caption{\label{bpart}The circular representation of the noncrossing partition
of type B $(1,-3,-6),(-1,3,6),(2,-2),(4),(-4),(5),(-5)$.}
\end{figure}
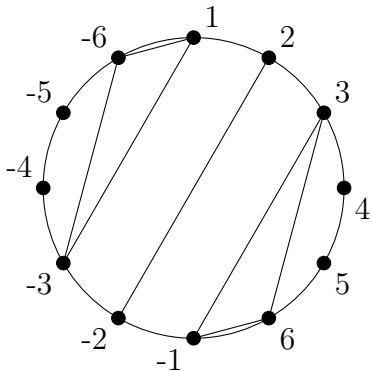

\medskip
The cardinality of $NC^B(n)$ is given by the central binomial coefficient
$\binom{2n}{n}$ and the number of elements of rank $k$ is given by
$\binom{n}{k}^2$ which are the type B ($m=0$) Narayana numbers.

\subsection{Gamma-positivity of type B ($m=0$) Narayana polynomials}

In order to prove the $\gamma$-positivity of the classical Narayana
polynomials, Simion and Ullman use a bijection between noncrossing partitions
and a set of words equivalent to 2-colored Motzkin words. We use a similar
method to prove the type B case.

\begin{definition}
A {2-path} of length $n$ is a lattice path between $(0,0)$ and $(0,n)$
that uses diagonal up steps $(1,1)$, diagonal down steps $(1,-1)$ and two
horizontal steps $(1,0)$, one being wavy, the other being straight. In other
words, a 2-path is a 2-Motzkin path that does not have the constraint of
remaining above the $x$-axis. We construct a map between $NC^B(n)$ and the set
of 2-paths of length $n$ in the following way.
\end{definition}

\medskip
Given the circular representation of a noncrossing type B partition $\pi$, for
each nonzero block with more than one edge, remove the longest edge. This does
not change the partition since this edge is the result of transitivity of the
other edges of its block. For each remaining edge, direct it according to the
shortest clockwise path between its two endpoints: if an edge has $i$ and $j$
as its endpoints and the clockwise path from $i$ to $j$ is shorter than the
clockwise path from $j$ to $i$ then the edge is directed from $i$ to $j$. If
both paths have the same length, replace it by two arcs going in both
directions.

\medskip
Now we build $\phi(\pi)=w_{1}w_{2}...w_{n}$ the 2-path associated with $\pi$
where for $1 \leq i \leq n$ 
\begin{itemize}
\item[(i)] If $i$ is the starting point of an arc that does not go to $i+1$
(or $-1$ if $i=n$), then $w_{i}$ is an up diagonal step,
\item[(ii)] If $i+1$ (or $-1$ if $i=n$) is the endpoint of an arc not coming
from $i$, then $w_{i}$ is a down diagonal step,
\item[(iii)] If there is an arc between $i$ and $i+1$ (or $-1$ if $i=n$), then
$w_{i}$ is a wavy horizontal step,
\item[(iv)] If there is no arc starting from $i$ and there is no arc ending in
$i+1$ (or $-1$ if $i=n$), then $w_{i}$ is a straight horizontal step.
\end{itemize}

\def\f {2} \def\d {7} \def\dy {-5}
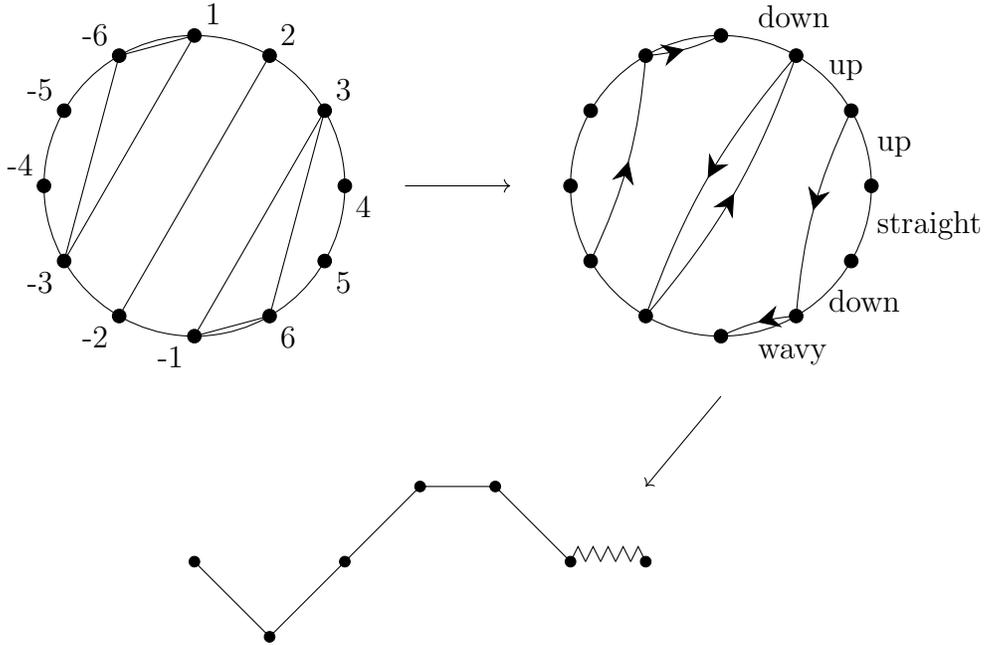
\begin{figure}[ht]
\begin{center}
\begin{tikzpicture}[scale=1]
\draw (0,0) circle (\f cm);
\foreach \a in {1,2,3}
		\pgfmathsetmacro\x{int(4-\a))}
        \filldraw ({\f*cos(30*\a)}, {\f*sin(30*\a)}) circle (0.09cm) node[above right](tr){\x};
\foreach \a in {-2,-1,0}
		\pgfmathsetmacro\x{int(4-\a))}
        \filldraw ({\f*cos(30*\a)}, {\f*sin(30*\a)}) circle (0.09cm) node[below right](sr){\x};
\foreach \a in {4,5,6}
		\pgfmathsetmacro\y{int(\a-10))}
        \filldraw ({\f*cos(30*\a)}, {\f*sin(30*\a)}) circle (0.09cm) node[above left] (ur){\y};
\foreach \a in {7,8,9}
		\pgfmathsetmacro\y{int(\a-10))}
        \filldraw ({\f*cos(30*\a)}, {\f*sin(30*\a)}) circle (0.09cm) node[below left](vr){\y};
\def\i{3}
\def\j{7}
\draw ({\f*cos(30*\i)}, {\f*sin(30*\i)})--({\f*cos(30*\j)}, {\f*sin(30*\j)});
\draw ({\f*cos(30*(\i+6))}, {\f*sin(30*(\i+6))})--({\f*cos(30*(\j+6))}, {\f*sin(30*(\j+6))});
\def\j{4}
\draw ({\f*cos(30*\i)}, {\f*sin(30*\i)})--({\f*cos(30*\j)}, {\f*sin(30*\j)});
\draw ({\f*cos(30*(\i+6))}, {\f*sin(30*(\i+6))})--({\f*cos(30*(\j+6))}, {\f*sin(30*(\j+6))});
\def\i{7}
\draw ({\f*cos(30*\i)}, {\f*sin(30*\i)})--({\f*cos(30*\j)}, {\f*sin(30*\j)});
\draw ({\f*cos(30*(\i+6))}, {\f*sin(30*(\i+6))})--({\f*cos(30*(\j+6))}, {\f*sin(30*(\j+6))});
\def\i{2}
\def\j{8}
\draw ({\f*cos(30*\i)}, {\f*sin(30*\i)})--({\f*cos(30*\j)}, {\f*sin(30*\j)});


\draw (\d,0) circle (\f cm);
\foreach \a in {-2,-1,...,9}
        \filldraw ({\d+\f*cos(30*\a)}, {\f*sin(30*\a)}) circle (0.09cm);
        
\begin{scope}[decoration={
    markings,
    mark=at position 0.5 with {\arrow{Stealth[length=3mm, width=3mm]}}}
    ]
	\def\i{3}
	\def\j{4}
	\draw [postaction={decorate}]({\d+\f*cos(30*\j)}, {\f*sin(30*\j)})to[bend right=10] ({\d+\f*cos(30*\i)}, {\f*sin(30*\i)});
\draw  [postaction={decorate}] ({\d+\f*cos(30*(\j+6))}, {\f*sin(30*(\j+6))})to[bend right=10]({\d+\f*cos(30*(\i+6))}, {\f*sin(30*(\i+6))});
\def\i{7}
\draw [postaction={decorate}]({\d+\f*cos(30*\i)}, {\f*sin(30*\i)})to [bend right=10]({\d+\f*cos(30*\j)}, {\f*sin(30*\j)});
\draw[postaction={decorate}]({\d+\f*cos(30*(\i+6))}, {\f*sin(30*(\i+6))})to [bend right=10]({\d+\f*cos(30*(\j+6))}, {\f*sin(30*(\j+6))});
\def\i{2}
\def\j{8}
\draw  [postaction={decorate}]({\d+\f*cos(30*\i)}, {\f*sin(30*\i)})to [bend right=10]({\d+\f*cos(30*\j)}, {\f*sin(30*\j)});
\draw  [postaction={decorate}]({\d+\f*cos(30*\j)}, {\f*sin(30*\j)})to [bend right=10]({\d+\f*cos(30*\i)}, {\f*sin(30*\i)});
\end{scope}

\draw ({\d+\f*cos(80)}, {\f*sin(80)}) node[above right]{down};
\draw ({\d+\f*cos(50)}, {\f*sin(50)}) node[right]{up};
\draw ({\d+\f*cos(15)}, {\f*sin(15)}) node[right]{up};
\draw ({\d+\f*cos(-15)}, {\f*sin(-15)}) node[right]{straight};
\draw ({\d+\f*cos(-50)}, {\f*sin(-50)}) node[right]{down};
\draw ({\d+\f*cos(-80)}, {\f*sin(-80)}) node[below right]{wavy};


\filldraw[black] (0,\dy) circle (2pt);
\filldraw (0,\dy) -- ++ (1,-1)circle (2pt) -- ++ (1,1)circle (2pt)
-- ++ (1,1)circle (2pt)-- ++ (1,0)circle (2pt)-- ++ (1,-1)circle (2pt);
\foreach \a in {0,2,...,8}
	\draw ({5+(\a)/10},\dy) -- ({5+(\a+1)/10},{\dy+1/5})--({5+(\a+2)/10},\dy);
\filldraw[black] (6,\dy) circle (2pt);


\draw[->] (2.8,0) -- (\d-2.8,0);
\draw[->] (\d,-2.8) -- (6,\dy+1);
\end{tikzpicture}
\end{center}
\caption{\label{Upsilon}Illustration of the bijection $\phi$ on an example in
$NC^B(6)$.}
\end{figure}

\begin{lem}
Let $n \geq 1$. The map $\phi$ is well-defined and it is a bijection between
$NC^B(n)$ and the set of $2$-paths of length $n$.
\end{lem}

As in type A, this bijection gives a decomposition of $NC^B$ into boolean
lattices resulting in the following formula.

\begin{state}
Let $n \geq 1$. We have
\begin{equation}
\sum_{i=0}^{n}N^{(0)}_{n,i} t^i
 = \sum_{k=0}^{\lfloor n/2 \rfloor} \binom{n}{2k}\binom{2k}{k}t^k(1+t)^{n-2k}.
\end{equation}
\end{state}

\begin{proof}
Let $n\geq 1$ and let $P(n)$ be the set of 2-paths of length $n$. The
bijection $\phi$ allows $P(n)$ to inherit a structure of ranked lattice from
the lattice of $NC^B(n)$.

Let $p \in P(n)$.
One easily checks that the number of up steps plus the number of wavy steps in
$p$ is equal to the rank of $p$ in the lattice.
Now let $p$ and $p'$ be in $P(n)$. We define the equivalence relation $p \sim
p'$ if $p$ and $p'$ have same set of up steps and same set of down steps.
Let $p=p_1p_2 \dots p_{n-1}$ and $p'=p'_1p'_2 \dots p'_{n-1}$ such
that $p \sim p'$. Then $p'$ covers $p$ if and only if there is $1 \leq i <n$
such that $p_i$ is a straight horizontal step, $p'_i$ is a wavy step and
$p_j=p'_j$ for $j \neq i$. This means that if $C$ is an equivalence class for
this relation then $C$ is a boolean lattice with its minimal element being a
$2$-path without any wavy steps and its maximal element being the same path
with all the straight steps replaced by wavy steps. Let $p$ be such a minimal
element, and let $k$ be its number of up steps. Then $p$ is of rank $k$ and
the maximal element of its class is of rank $n-k$.

Therefore the type B Narayana polynomials can be written as a sum of $\gamma_k
t^k(1+t)^{n-1-2k}$ where $\gamma_k$ is the number of $2$-paths of length $n$
with $k$ up steps and not any wavy steps. Such paths have $\binom{n}{2k}$
possibilities for the positions of the up and down steps and $\binom{2k}{k}$
configurations for each of those positions, so
$\gamma_k=\binom{n}{2k}\binom{2k}{k}$ which concludes the proof of the
formula.
\end{proof}

\subsection{The type D case}

The analogue of noncrossing partitions in type D can be defined as a
sublattice of the noncrossing partitions in type B. In fact, a type D
partition is a type B partition in which the zero block is not made of a
single pair $\{-i,i\}$ and $NC^D$ is the sublattice of $NC^B$ induced by these
partitions.
Let $i \leq n$. Reiner also defines $NC^B(n,i)$ as the sublattice of $NC^B$
consisting of the partitions with zero-block $\{-i,i\}$ and gives the
following decomposition of the lattice.

\begin{state}[Proposition 9 in \cite{Reiner_1997}]
\label{dec}
Let $n \geq 1$, we have

\begin{equation}
NC^B(n) = \bigsqcup_{i \leq  n} NC^B(n,i) \sqcup NC^D(n)
\end{equation}

with $NC^B(n,i) \cong NC(n-1)$.
\end{state}

Type D analogues of Catalan and Narayana numbers follows from this
decomposition:

\begin{equation}
Cat^D(n)= \binom{2n}{n} - \binom{2(n-1)}{n-1},
\end{equation}

\begin{equation}
N^D(n,k) = \binom{n}{k}^2 - \frac{n}{n-1} \binom{n-1}{k}\binom{n-1}{k-1}.
\end{equation}

The decomposition of $NC^B$ into boolean lattices given in the previous
Section is a refinement of the one given by Proposition~\ref{dec}. Indeed, let
$1\leq i\leq n$ and $\pi \in NC^B(n,i)$ then its circular representation has
an edge between $i$ and $-i$ which means that its corresponding 2-path has an
up step in position $i$ and a down step in position $i-1$ (or $n$ if $i=1$).
Since all the elements in the class of $\pi$ have the same up and down steps,
they all are in $NC^B(n,i)$. The number of minimal elements of rank $k$ in
$NC^B(n,i)$ is $\binom{n-2}{2(k-1)}C_{k-1}$ since they are the ones with $k$
up steps including a down step in position $i-1$ and an up step in position
$i$. This leads to the $\gamma$-positivity formula for type D Narayana
numbers:

\begin{equation}
\sum_{i=0}^{n}N^D(n,i)t^i
 = \sum_{k=0}^{\lfloor n/2 \rfloor}
  \left(\binom{n}{2k}\binom{2k}{k}-n\binom{n-2}{2(k-1)}C_{k-1} \right)
  t^k(1+t)^{n-2k}.
\end{equation}

\section{A $q$-analog of the $m$-Narayana numbers}
\label{sec:qNar}

We now turn to $q$-Narayana numbers and their generalizations to all $m$. 
As mentioned before, the definitions and some properties we establish here
were presented in a talk from Wachs (see~\cite{Wachs}) that detailed a yet
unpublished joint paper with Krattenthaler~\cite{K}. Our own interest in the subject
and our desire to emphasize the connection inside the whole family of super
Catalan numbers made us look for our own proofs of some statements of their
paper.

The following are known $q$-Narayana refinements for the $q$-central binomial
coefficients and the $q$-Catalan numbers, cases $m=0$ and $m=1$ respectively:
\begin{align}
& T_{n,0}(q) = \qbinom{2n}{n}
  = \sum_{k=0}^n q^{k^2} \qbinom{n}{k}^2  \\
& T_{n,1}(q) = (q+1) C_n(q)
  = \sum_{k=0}^{n -1}q^{k(k+1)}
      \frac{(1-q^2)}{1-q^n} \qbinom{n}{k} \qbinom{n}{k+1}  \qquad
\text{with $n\geq1$}.
\end{align}
More generally, we observe that every member of the super Catalan family
enjoys a similar decomposition.

\begin{state}[\cite{Wachs}]
Let $m$ and $n$ be two nonnegative integers such that $m\leq n$.
Then the $q$-super Catalan numbers decompose in a Narayana-like way as
\begin{equation}
\label{eq:Narade}
T_{n,m}(q)
 = \qbinom{2m}{m}\qbinom{n}{m}^{-1}
   \sum_{k=0}^{n -m} q^{k(k+m)} \qbinom{n}{k} \qbinom{n}{k+m}.
\end{equation}
\end{state}

\begin{proof}
Immediate by application of the $q$-Vandermonde identity:
\begin{align*}
& \qbinom{2m}{m}\qbinom{n}{m}^{-1}
  \sum_{k=0}^{n -m} q^{k(k+m)} \qbinom{n}{k} \qbinom{n}{k+m}
  =  \qbinom{2m}{m}\qbinom{n}{m}^{-1}
     \sum_{k=0}^{n -m} q^{k(k+m)}  \qbinom{n}{k} \qbinom{n}{n - m - k} \\
& = \qbinom{2m}{m}\qbinom{n}{m}^{-1} \qbinom{2n}{n-m}
  = \frac{\qnum{2m}! \qnum{n-m}!}{\qnum{m}! \qnum{n}!}
    \frac{\qnum{2n}!}{\qnum{n-m}! \qnum{n+m}!}
  = T_{n,m}(q).
\end{align*}
\end{proof}

Therefore the following definition makes sense:
\begin{definition}[\cite{Wachs}]
Let $m$, $n$, and $k$ be three integers, such that $m\leq n$
and $k\leq n-m$.
The $q$-Narayana numbers of type $m$, or simply the $q$-$m$-Narayana numbers,
are
\begin{equation}
\label{eq:N_def}
N^{(m)}_{n,k}(q)
 = \qbinom{2m}{m}\qbinom{n}{m}^{-1} \qbinom{n}{k} \qbinom{n}{k+m},
\end{equation}
or, equivalently,
\begin{equation}
\label{eq:N_hdef}
N^{( m )}_{ n,k} ( q ) = \frac{( q^{m +1} ; q )_m ( q^{n -k +1} ; q )_k (
q^{n- k -m +1} ; q )_{n + k}}{ ( q^{n -m +1} ; q )_m ( q ; q )_k ( q ; q )_{n
+ k}}.
\end{equation}
\end{definition}

We shall prove in Section~\ref{sec-Na_pos}, Proposition~\ref{state:positivity}
that the numbers $N^{(m)}_{n,k}(q) $ are $q$-positive integers.

\subsection{The $q$-Kreweras Identity for $m$-Narayana numbers}

To deepen the analogy between $m$-Narayana numbers and their Catalan
analogs, we shall prove a version of the $q$-Kreweras formula satisfied by
the $m$-Narayana numbers.
Recall that the Kreweras identity~\cite{Krew} for Narayana numbers reads
in our notation (case $m=1$):

\begin{equation}
N^{(1)}_{n+k+1,k} = \sum_{s=0}^k N^{(1)}_{n,k-s} \binom{2n+ s}{s},
\end{equation}
which also reads
\begin{equation}
\frac1{n+k+1} \binom{n+k+1}{k} \binom{n+k+1}{k+1}
= \sum_{s=0}^k \frac1{n} \binom{n}{k-s} \binom{n}{k-s+1} \binom{2n+ s}{s}.
\end{equation}
We shall prove its $q$-extension to all $m$-Narayana numbers:

\begin{state}
Let $m$, $n$, and $k$ be three integers such that $m\leq n$
and $k\leq n-m$.
Then the $q$-$m$-Narayana numbers satisfy the following
generalized $q$-analog of the Kreweras identity:
\begin{align}
\label{eq:m-Kreweras}
& N^{(m)}_{n+k+m,k}(q)
= \sum_{s=0}^k q^{(k-s)(k+m-s)} N^{(m)}_{n,k-s}(q) \qbinom{2n+s}{s}
   \quad \text{or, equivalently:} \\
& N^{(m)}_{n+k+m,k}(q) 
= \sum_{s=0}^k q^{s(s+m)} N^{(m)}_{n,s}(q) \qbinom{2n+k-s}{2n}.\notag
\end{align}
\end{state}

\begin{proof}
We will need to use $q$-Pfaff–Saalsch\"{u}tz summation \cite{GR} (II.12),
which we rewrite here:
\begin{equation}
\label{eq:Saal}
\sum_{k=0}^{\infty}
 \frac{(a;q)_k (b;q)_k (q^{-N};q)_k}{(q;q)_k (c;q)_k ( d;q)_k} q^k
= \frac{\left(\frac{c}{a};q\right)_N \left(\frac{c}{b};q\right)_N}{(c;q)_N \left(\frac{c}{ab};q\right)_N}
  \quad \text{if} \quad d = \frac{a b q^{1\!-\!N}}{c}.
\end{equation}
We need to transform the right-hand side of \eqref{eq:m-Kreweras} to the form where the
Formula \eqref{eq:Saal} can be applied.
We consider a generic term in the right-hand side of
Equation~\eqref{eq:m-Kreweras} :
\[
\frac{(q^{m+1}; q)_m}{(q^{n - m +1} ;q)_m}  \qbinom{n}{k-s}
\qbinom{n}{k-s+m}\qbinom{2n +s }{s},
\]
and deal with each binomial coefficient in turn (the factor $\frac{(q^{m+1}; q)_m}{(q^{n - m +1} ;q)_m} $ is summation-index independent and can be ignored for now):
\begin{align}
\label{eq:Ka1}
&   \qbinom{n}{k-s} = \frac{(q^{n - k +s +1};q)_{k-s} }{(q;q)_{k-s}}  = \frac{(q^{n-k+1};q)_s (q^{n - k +s +1};q)_{k-s} (q^{k - s+1};q)_s}{(q^{n-k+1};q)_s (q;q)_{k-s}(q^{k - s+1};q)_s } \\
\label{eq:Ka2}
& =(-1)^s q^{\frac{s(2k-s+1)}{2}}  \frac{(q^{n-k+1};q)_{k  } }{(q;q)_k} \frac{(q^{-k};q)_s}{ (q^{n-k+1};q)_s} 
\end{align}
where
\begin{itemize}
\item in \eqref{eq:Ka1}, we have multiplied both the numerator and the
denominator by $(q^{n-k+1};q)_s $ and by $(q^{k - s+1};q)_s$;
\item in \eqref{eq:Ka2}, we have used \eqref{eq:Kone} twice to simplify the
products $(q^{n-k+1};q)_s (q^{n - k +s +1};q)_{k-s}$ and $(q;q)_{k-s}(q^{k
- s+1};q)_s $ as well as reversing the powers in $  (q^{k - s+1};q)_s $ using
\eqref{eq:Kfour};
\end{itemize}

We also rewrite
\begin{align}
\label{eq:Ka3}
&  \qbinom{n}{k-s+m} = \frac{(q^{n-k+s-m+1};q)_{k-s+m} }{(q;q)_{k - s +m}} =  \frac{(q^{n-k-m+1};q)_s(q^{n-k+s-m+1};q)_{k-s+m}(q^{k -s +m +1};q)_s}{(q^{n-k-m+1};q)_s(q;q)_{k - s +m}(q^{k -s +m +1};q)_s} \\
\label{eq:Ka4}
&  = (-1)^s q^{\frac{s(2k + 2m - s +1)}{2}}  \frac{(q^{n - k -m+1};q)_{k +m } }{(q;q)_{k+m}} \frac{(q^{-k -m};q)_s}{(q^{n -k -m+1};q)_s} 
\end{align}
where
\begin{itemize}
\item in \eqref{eq:Ka3}, we have multiplied both the numerator and the
denominator by $(q^{n-k-m+1};q)_s $ and by $(q^{k -s +m +1};q)_s$;
\item in \eqref{eq:Ka4}, we have used \eqref{eq:Kone} twice to simplify the
products $(q^{n-k-m+1};q)_s(q^{n-k+s-m+1};q)_{k-s+m}$ and $ (q;q)_{k - s
+m}(q^{k -s +m +1};q)_s$ as well as reversing the powers in $ (q^{k -s +m
+1};q)_s $ using \eqref{eq:Kfour};
\end{itemize}
and, finally,
\[
\qbinom{2n +s }{s} = \frac{(q^{2n+1};q)_s}{(q;q)_s}.
\]
Collecting all the factors of a generic term, we get:
\begin{align}
\label{eq:krew-trans}
&  \frac{(q^{m+1}; q)_m}{(q^{n - m +1} ;q)_m} \qbinom{n}{k-s} \qbinom{n}{k-s+m} \qbinom{2n +s }{s}  \\
& = \frac{(q^{m+1}; q)_m (q^{n-k+1};q)_k (q^{n - k -m+1};q)_{k +m }}{(q^{n - m +1} ;q)_m(q;q)_k (q;q)_{k+m}} q^{s(-s +2k + m +1)} \frac{(q^{-k};q)_s (q^{-k -m};q)_s (q^{2n+1};q)_s}{(q;q)_s (q^{n-k+1};q)_s(q^{n -k -m+1};q)_s}. \notag
\end{align}
Then, getting rid of extra powers of $q$ by multiplying each term of the sum by
$q^{-k(k+m)}$ and factoring out summation-index independent terms, the
sum becomes:
\begin{align}
& \sum_{s=0}^k q^{-s(-s+2k +m)}  \frac{(q^{m+1}; q)_m}{(q^{n - m +1} ;q)_m} \qbinom{n}{k-s} \qbinom{n}{k-s+m} \qbinom{2n +s }{s} \\
\label{eq:Krsum}
& = \frac{(q^{m+1}; q)_m (q^{n-k+1};q)_k (q^{n - k -m+1};q)_{k +m }}{(q^{n - m +1} ;q)_m(q;q)_k (q;q)_{k+m}} \sum_{s=0}^k \frac{(q^{-k};q)_s (q^{-k -m};q)_s (q^{2n+1};q)_s}{(q;q)_s (q^{n-k+1};q)_s(q^{n -k -m+1};q)_s} q^s,
\end{align}
and now we can apply the $q$-Pfaff–Saalsch\"{u}tz summation \eqref{eq:Saal}
with:
\begin{equation}
 a= q^{-k -m}, \quad
 b = q^{2n+1}, \quad
 c = q^{n-k+1}, \quad
 d = q^{n - k - m +1} \quad \text{as:} \quad
 \frac{a b q^{1 - k}}{c} =  q^{n - k -m +1}= d
\end{equation}
so that
\begin{equation}
\sum_{s=0}^k \frac{(q^{-k};q)_s (q^{-k -m};q)_s (q^{2n+1};q)_s}{(q;q)_s (q^{n-k+1};q)_s(q^{n -k -m+1};q)_s}  q^s   = \frac{(q^{n+m+1};q)_k (q^{-n -k};q)_k}{(q^{n-k+1};q)_k(q^{-n+m};q)_k}.
\end{equation}
We now prove that the resulting expression is indeed $N^{(m)}_{n+k+m,k}(q)$. After the summation above, the expression \eqref{eq:Krsum} takes on the form:
\begin{align}
& \frac{(q^{m+1}; q)_m (q^{n-k+1};q)_{k - m } (q^{n - k -m+1};q)_{k +m }}{(q;q)_k (q;q)_{k+m}}  \frac{(q^{n+m+1};q)_k (q^{-n -k};q)_k}{(q^{n-k+1};q)_k(q^{-n+m};q)_k} \\
\label{eq:Kr1}
& = q^{-k(k+m)} \frac{(q^{m+1}; q)_m (q^{n-k+1};q)_{k - m } (q^{n - k -m+1};q)_{k +m }}{(q;q)_k (q;q)_{k+m}}  \frac{(q^{n+m+1};q)_k (q^{n +1};q)_k}{(q^{n-k+1};q)_k(q^{n-k-m+1};q)_k}  \\
\label{eq:Kr2}
& = q^{-k(k+m)}  \frac{(q^{m+1}; q)_m (q^{n-k+1};q)_{k - m } }{(q;q)_k (q;q)_{k+m}} \frac{(q^{n+m+1};q)_k (q^{n-m+1};q)_{k+m}}{(q^{n-k+1};q)_k}  \\
\label{eq:Kr3}
& =  q^{-k(k+m)}  \frac{(q^{m+1}; q)_m  }{(q;q)_k (q;q)_{k+m}} \frac{(q^{n+m+1};q)_k (q^{n-m+1};q)_{k+m}}{(q^{n-m+1};q)_m}  \\
\label{eq:Kr4}
& =  q^{-k(k+m)}  \frac{(q^{m+1}; q)_m}{(q^{n+k+1};q)_m} \frac{ (q^{n+m+1};q)_k (q^{n+1};q)_k (q^{n+k+1};q)_m}{(q;q)_k (q;q)_{k+m} }  \\
\label{eq:Kr5}
& =  q^{-k(k+m)}  \frac{(q^{m+1}; q)_m}{(q^{n+k+1};q)_m} \frac{
(q^{n+m+1};q)_k (q^{n+1};q)_{k+m}}{(q;q)_k (q;q)_{k+m} }  q^{-k(k+m)}
N^{(m)}_{n+k+m,k}(q),
\end{align}
where 
\begin{itemize}
\item in \eqref{eq:Kr1}, we have applied \eqref{eq:Kfour} to both factors
$(q^{-n -k};q)_k$ and $(q^{-n+m};q)_k$;
\item in \eqref{eq:Kr2}, we have combined the three terms $(q^{n-k-m+1};q)_{k
+ m } , (q^{n-k-m+1};q)_k $, and $(q^{n +1};q)_k$  applying \eqref{eq:Kone};
\item in \eqref{eq:Kr3}, we have combined $(q^{n-k+1};q)_{k - m }$ and $
(q^{n-k+1};q)_k$ applying \eqref{eq:Kone};
\item in \eqref{eq:Kr4}, we have combined $ (q^{n-m+1};q)_{k+m} $ and
$(q^{n-m+1};q)_m $ applying \eqref{eq:Kone} and multiplied both the numerator
and the denominator by $(q^{n+k+1};q)_m $;
\item in \eqref{eq:Kr5}, we have combined $(q^{n+1};q)_k$ and $
(q^{n+k+1};q)_m$ using \eqref{eq:Kone}, which allowed us to recognize the
expression of $N^{(m)}_{n+k+m,k}(q)$ as in \eqref{eq:N_hdef}.
\end{itemize}
Rearranging the powers of $q$ gives Formula \eqref{eq:m-Kreweras} as
stated.
\end{proof}

\begin{cor}
Let $m$, $n$, and $k$ be three integers such that $m\leq n$ and $k\leq n-m$.
We then have the following extension of the Kreweras identity to all super
Catalan numbers:
\begin{equation}
N^{(m)}_{n+k+m,k} =  \sum_{s=0}^k N^{(m)}_{n,k-s} \binom{2n+s}{s}.
\end{equation}
\end{cor}

\subsection{The Le Jen-Shoo identity for $m$-Narayana numbers}

Among known relations that can be interpreted as relations between the
Narayana numbers, the Le Jen-Shoo identity, (6.32) of \cite{Gould}
\begin{equation}
\label{eq:LeJenShoo}
\binom{n+k}{k}^2 = \sum_{s=0}^k \binom{k}{s}^2 \binom{n+2k-s}{2k} 
\end{equation}
is very interesting since it relates several Narayana numbers at $m=0$.

It has the following equivalent form for all $m$-Narayana numbers:

\begin{state}
Let $m$, $n$, and $k$ be integers such that $m\leq n$ and $k\leq n-m$
with the extra condition $k\geq m$.
We then have
\begin{equation}
\label{eq:Riordan}
N^{(m)}_{n,k}(q)
 = \sum_{s=0}^{k-m} q^{s(s+m)} N^{(m)}_{k,s}(q) \qbinom{n+k -s-m}{2k}.
\end{equation}
\end{state}

Note that the condition $k\geq m$ is needed to make sure that the $q$-Narayana
polynomial $N^{(m)}_{k,s}$ exists.

\begin{proof}
Let us take a generic term in the sum and consider each factor separately.
First we rewrite $N^{(m)}_{k,s}(q)$ in the appropriate hypergeometric form:
\begin{align}
&  \qbinom{k}{s}
  = (-1)^s q^{\frac{s(2k-s+1)}{2}} \frac{(q^{-k};q)_s}{(q;q)_s} \notag \\
& \frac1{(q^{k - m +1} ;q)_m} \qbinom{k}{s+m}
 = \frac{(q^{m+1};q)_{k-m}}{(q;q)_k} (-1)^s q^{\frac{s(2k - 2m -s+1)}{2}}
   \frac{(q^{-k+m};q)_s}{(q^{m+1};q)_s} \notag \\
\label{eq:m_Nar_hyper}
&  \frac{(q^{m+1}; q)_m}{(q^{k - m +1} ;q)_m} \qbinom{k}{s} \qbinom{k}{s+m}
 = q^{s(2k-s-m+1)}  \frac{(q^{m+1}; q)_m  }{ (q;q)_m}
   \frac{(q^{-k};q)_s (q^{-k+m};q)_s }{(q;q)_s  (q^{m+1};q)_s}
\end{align}
and then rewrite the binomial as
\begin{equation}
 \qbinom{n+k-s-m}{2k} = q^{-2ks}
 \frac{(q^{2k+1};q)_{n-k-m}(q^{-n+k+m};q)_s}{(q;q)_{n-k-m}(q^{-n-k+m};q)_s}.
\end{equation}

Multiplying each term by $q^{s(s+m)}$ to remove unwanted powers of $q$, the
sum becomes:
\begin{align*}
& \sum_{s \geq 0} q^{s(s+m)}
  \frac{(q^{m+1}; q)_m}{(q^{k - m +1} ;q)_m} \qbinom{k}{s} \qbinom{k}{s+m}
  \qbinom{n+k-s-m}{2k}  \\
& = \sum_{s \geq 0}
    \frac{(q^{m+1}; q)_m}{(q;q)_m}
    \frac{(q^{-k};q)_s (q^{-k+m};q)_s }{(q;q)_s  (q^{m+1};q)_s}
    \frac{ (q^{2k+1};q)_{n-k-m} (q^{-n+k+m};q)_s}{ (q;q)_{n-k-m} (q^{-n-k+m};q)_s} q^{s} \\
& = \frac{(q^{m+1}; q)_m  }{ (q;q)_m}
    \frac{(q^{2k+1};q)_{n-k-m} }{(q;q)_{n-k-m}}
    \sum_{s \geq 0} \frac{(q^{-k};q)_s (q^{-k+m};q)_s }{(q;q)_s (q^{m+1};q)_s}
       \frac{ (q^{-n+k+m};q)_s}{  (q^{-n-k+m};q)_s} q^{s}
\end{align*}
and we can use the $q$-Pfaff–Saalsch\"{u}tz summation with $N=k-m$ and
\begin{equation}
a=q^{-k}, \quad
b= q^{-n+k+m}, \quad
c = q^{m+1}, \quad
d = \frac{a b}{c} q^{1-N} = q^{-k - n+ k +m- m-1 + 1 -k +m} = q^{-n +m - k}.
\end{equation}

We then get
\begin{align*}
& \sum_{s=0}^{k-m} q^{s(s+m)}
    \frac{(q^{m+1}; q)_m}{(q^{k - m +1} ;q)_m}
    \qbinom{k}{s} \qbinom{k}{s+m} \qbinom{n+k-s-m}{2k} \\
& = \frac{(q^{m+1}; q)_m  }{ (q;q)_m}
    \frac{(q^{2k+1};q)_{n-k-m} }{(q;q)_{n-k-m}}
    \frac{(q^{m+k+1};q)_{k-m} (q^{n-k+1};q)_{k-m}}{(q^{m+1};q)_{k-m} (q^{n+1};q)_{k-m}}.
\end{align*}

To unravel the right-hand side, we start by simplifying
\begin{align*}
& \frac{(q^{m+k+1};q)_{k-m} (q^{2k+1};q)_{n-k-m}}{(q;q)_{n-k-m}(q^{n+1};q)_{k-m}}
 =  \qbinom{n}{k+m}, \\
\end{align*} 
where we have used \eqref{eq:Kone} first to combine the terms in the numerator
and then divide out $(q^{n+1};q)_{k-m}$. \\
There only remains
\begin{align}
\label{eq:ch1}
& \frac{(q^{m+1}; q)_m  }{ (q;q)_m} \frac{(q^{n-k+1};q)_{k-m}}{(q^{m+1};q)_{k-m}}
= \frac{(q^{m+1}; q)_m (q^{n-k+1};q)_{k-m}} {(q;q)_k} \\
\label{eq:ch2}
& = \frac{(q^{m+1}; q)_m  (q^{n-k+1};q)_{k-m} (q^{n - m +1} ;q)_m} {(q^{n - m +1} ;q)_m (q;q)_k}  = \frac{(q^{m+1}; q)_m  }{ (q^{n - m +1} ;q)_m} \qbinom{n}{k},
\end{align}
where in \eqref{eq:ch1} we have used \eqref{eq:Kone} to combine the terms in the denominator; then, in \eqref{eq:ch2}, multiplied the numerator and the denominator by $(q^{n - m +1} ;q)_m$ and then used  \eqref{eq:Kone} to combine two terms in the numerator. \\
Putting these together, we get:
\[
\sum_{s=0}^{k-m} q^{s(s+m)}
    \frac{(q^{m+1}; q)_m}{(q^{k - m +1} ;q)_m}
    \qbinom{k}{s} \qbinom{k}{s+m} \qbinom{n+k-s-m}{2k}  = \frac{(q^{m+1}; q)_m
}{ (q^{n - m +1} ;q)_m} \qbinom{n}{k} \qbinom{n}{k+m}.
\]
\end{proof}

We could not find in the literature the classical Narayana numbers version of
identity~\eqref{eq:Riordan}:

\begin{cor}
Let $n$ and $k$ be positive integers, such that $n\geq k+1$. Then
\begin{equation}
\label{eq:claRiordan}
N_{n,k} = \sum_{s=0}^{k-1} N_{k,s}  \binom{n+k-s-1}{2k}.
\end{equation}
\end{cor}

\subsection{The positivity of the $q$-$m$-Narayana polynomials}
\label{sec-Na_pos}

Let us recall the definition of the $q$-analogs of the $m$-Narayana numbers
seen in~\eqref{eq:N_def}:
\[
N^{(m)}_{n,k}(q)
= \qbinom{2m}{m}\qbinom{n}{m}^{-1} \qbinom{n}{k} \qbinom{n}{k+m},
\]
where $n$, $m$,and $k$ are nonnegative integers such that $m\leq n$ and
$k\leq n-m$.
While the fact that $N^{(m)}_{n,k}(q)$ are monic and palindromic immediately
follows from the definition, the integrality and positivity is far from
obvious but it is true.

\begin{state}[\cite{Wachs}]
\label{state:positivity}
Let $n$, $m$, and $k$ be nonnegative integers such that $m\leq n$ and
$k\leq n-m$.
The $q$-analogs of the $m$-Narayana numbers are monic palindromic polynomials
in $q$ with positive integer coefficients.
\end{state}

Our proof is based on the following identity, which we already
encountered in Section~\ref{sec:gamma} at $q=1$:
\begin{state}
Let $n$, $m$, and $k$ be nonnegative integers such that $m\leq n$ and
$k\leq\fl{n-m}$.
\begin{equation}
\label{eq:Naq-pos}
N^{(m)}_{n,k}(q) = \sum_{s=0}^{k} q^{s(s+m)}
    \qbinom{n-m}{2s} T_{m,s}(q) \qbinom{n - 2s-m}{k-s}.
\end{equation}

Surprisingly here, we do not impose the usual condition on the indices of the
super Catalan, \emph{i.e.}, $m$ need not be greater than or equal to $s$ in
the right-hand side.
\end{state}

\begin{proof}
The calculation is simple enough so we will directly use the $q$-Vandermonde
formula at the last step of the calculation \eqref{eq:pos-last}:
\begin{align}
&  \sum_{s=0}^{k} q^{s(s+m)}
    \qbinom{n-m}{2s} \frac{\qnum{2m}! \qnum{2s}!}{\qnum{m}! \qnum{s}! \qnum{s+m}!} \qbinom{n - 2s-m}{k-s} \\
&  = \frac{\qnum{2m}!}{\qnum{m}!}
     \sum_{s=0}^{k} q^{s(s+m)}\frac{\qnum{n-m}!}{\qnum{s}! \qnum{s+m}! \qnum{k-s}! \qnum{n-k- s -m}!}\\
& = \frac{\qnum{2m}!}{\qnum{m}!}
    \sum_{s=0}^{k}  q^{s(s+m)} \frac{\qnum{n-m}! \qnum{m}!}{\qnum{n}!}
      \frac{\qnum{n}!}{\qnum{m}! \qnum{s}! \qnum{s+m}! \qnum{k-s}! \qnum{n-k- s -m}!}\\
& = \frac{\qbinom{2m}{m}}{\qbinom{n}{m}}
    \sum_{s=0}^{k} q^{s(s+m)} \frac{\qnum{n}!}{ \qnum{s}! \qnum{s+m}! \qnum{k-s}! \qnum{n-k- s -m}!}\\
& = \frac{\qbinom{2m}{m}}{\qbinom{n}{m}} \qbinom{n}{k+m}
    \sum_{s=0}^{k} q^{s(s+m)}\frac{\qnum{k+m}!}{\qnum{k-s}! \qnum{s+m}!}
      \frac{ \qnum{n - k -m}!}{ \qnum{s}! \qnum{n-k- s -m}!}\\
\label{eq:pos-last}
& = \frac{\qbinom{2m}{m}}{\qbinom{n}{m}} \qbinom{n}{k+m}
    \sum_{s=0}^{k} q^{s(s+m)}\qbinom{k+m}{k-s} \qbinom{n - k -m}{s}
  = \qbinom{2m}{m}\qbinom{n}{m}^{-1} \qbinom{n}{k+m} \qbinom{n}{k}.
\end{align}
\end{proof}

Identity~\eqref{eq:Naq-pos} proves the validity of
Proposition~\ref{state:positivity} for $k\leq \fl{n-m}$ as $m$-Narayana are
positive sums of positive integer polynomials: both $q$-binomial coefficients and the $q$-super Catalan numbers are positive polynomials (due to \cite{Warnaar2011}). To establish this fact for the upper range of $k$ up to $n-m$, we appeal to the symmetry of $m$-Narayana numbers:
\begin{equation}
\label{eq:sym}
N^{(m)}_{n,k}(q) = N^{(m)}_{n,n-m-k}(q),
\end{equation}
which is obvious from their definition (see Equation~\eqref{eq:N_def}).

Consequently, for $n$, $m$, and $k$ nonnegative integers such that $m\leq n$
and $k\leq \fl{n-m}$, we have
\begin{cor}
\begin{equation}
\label{eq:Na-pos}
N^{(m)}_{n,k}= \sum_{s=0}^{k} 
    \binom{n-m}{2s} T_{m,s} \binom{n - 2s-m}{k-s}.
\end{equation}
\end{cor}
This is the identity that was needed in Section~\ref{sec:gamma},
Equation~\eqref{eq:gamma-id}.

\subsection{Conjectures on the Narayana polynomials}

We will first focus on the $q$-analog of the $m$-Narayana numbers. They are
palindromic directly thanks to their definition as a quotient of $q$-binomial
coefficients that simplify into a polynomial.
Based on computer experiments, we also think that 

\begin{conjecture}
Let $n$, $m$, and $k$ be nonnegative integers such that $m\leq n$ and $k\leq
n-m$.

Then $N^{(m)}_{n,k}(q)$ is unimodal as a polynomial in $q$.
\end{conjecture}

Note that the $N^{(m)}_{n,k}(q)$ are not $\gamma$ positive. They even seem to
be quite the opposite: they seem to be alternating in sign.

Also based on computer experiments, we also conjecture that
\begin{conjecture}
Let $n$ and $m$ be nonnegative integers such that $m\leq n$.

Then the sequence of polynomials $\{N^{(m)}_{n,k}(q)\}_{0\leq k\leq n-m} $ is
$q$-log-concave for fixed $n$ and $m$,
where a sequence $\{a_n(q)\}_{0\leq k\leq N} $ is $q$-log-concave if
the sequence $a_k(q)^2 - a_{k-1}(q) a_{k+1}(q)$ has positive coefficients
for all $1\leq k\leq N-1$.
\end{conjecture}

\medskip
Returning to the $m$-Narayana polynomials \eqref{def:mNara}, let us recall
that a sequence of polynomials $\{p_n(t) \}_{n \geq 0} $ is $t$-log-convex if
for any $k\geq 1$, the difference
\[
p_{k+1}(t) p_{k-1}(t) - p_k^2(t)
\]
has nonnegative coefficients as a polynomial in $t$.

In both cases $m=0$ and $m=1$ (see~\cite{ChenB2010}, \cite{ChenA2010}), it is known that the
Narayana polynomials form a $t$-log-convex sequence.
We conjecture that this pattern is true for all $m$:

\begin{conjecture}
Let $m$ be a nonnegative integer.

The sequence of polynomials $\mathcal{N}^{(m)}_{n}(1,t)$ starting at $n=m$ forms a
$t$-log-convex sequence.
\end{conjecture}

\section{Acknowledgements}

Both authors would like to thank Jean-Christophe Novelli for a very thorough
reading of the manuscript and many useful suggestions and corrections. In
addition, the second author thanks Christian Krattenthaler for sharing his
unpublished notes. The authors are grateful to the referee for corrections, many insightful comments and, in particular, for suggestions on streamlining the proof of Proposition~\ref{prop:q-Touchard}.

\bibliographystyle{alpha}
\bibliography{SuperCatalan.bib}

\end{document}